\documentclass[12pt]{amsart}
\usepackage{amsmath,amsthm,amsfonts,amssymb}

\newtheorem{theorem}{Theorem}[section] 
\newtheorem{conj}{Conjecture}[section] 
\newtheorem{lemma}[theorem]{Lemma}

\newtheorem{prop}[theorem]{Proposition}
\newtheorem{remark}[theorem]{Remark}

\numberwithin{equation}{section}

\usepackage{color}

\begin{document}

\title{Scaling Properties of a Moving Polymer}
\author{Carl Mueller}
\address{Carl Mueller
\\Department of Mathematics
\\University of Rochester
\\Rochester, NY  14627}
\email{carl.e.mueller@rochester.edu}
\author{Eyal Neuman}
\address{Eyal Neuman
\\Department of Mathematics
\\Imperial College
\\London, UK}
\email{e.neumann@imperial.ac.uk}
\thanks{CM was partially supported by a Simons grant.} 
\keywords{polymer, scaling, heat equation, white noise, stochastic partial differential equations.}
\subjclass[2020]{Primary, 60H15; Secondary, 82D60.}

\begin{abstract} 
We set up an SPDE model for a moving, weakly 
self-avoiding polymer with intrinsic length $J$ taking values in 
$(0,\infty)$.  Our main result states that the effective radius of 
the polymer is approximately $J^{5/3}$; evidently for large $J$ 
the polymer undergoes stretching.  This contrasts with the 
equilibrium situation without the time variable, where many earlier 
results show that the effective radius is approximately $J$.  

For such a moving polymer taking values in $\mathbf{R}^2$, we offer a 
conjecture that the effective radius is approximately $J^{5/4}$.  
\end{abstract}

\maketitle

\section{Introduction}
\label{section:introduction}

Because of their widespread presence in the physical world, polymers 
have been intensively studied in chemistry, statistical mechanics, 
probability, and other fields.  See Doi and Edwards \cite{DE09} for a 
wide-ranging treatment from the physical point of view, and den 
Hollander \cite{dH09}, Giacomin \cite{Gia07}, and Bauerschmidt et. 
al. \cite{BDGS12} for rigorous mathematical results.   
K\"onig and van der Hofstad \cite{vdHK2001} discuss the one-dimensional case.

From the mathematical point of view, the study of polymers is 
hampered by the many complicated factors influencing their shapes.  
The simplest model for a polymer is a random walk where the time 
parameter of the walk represents the distance along the polymer.  In 
this model we assume that new segments are attached to the end of the 
polymer with a random orientation.  Perhaps the most important 
modification of this model is to penalize self-intersection; clearly 
two segments of the polymer cannot occupy the same position at the 
same time.  If self intersection is prohibited, then we are led to 
study self-avoiding random walks.  There is a large literature on 
this subject, see Madras and Slade \cite{MS13}. One feature of 
interest is the macroscopic extension of the polymer, and there are 
various ways to quantify this notion.  We give our own definition 
later, which we call the effective radius or simply the radius. The 
extension is often measured by the variance of the end-to-end 
distance, $E[|S_n|^2]$, where $S_n$ is the location of the polymer at 
$n$ units from its beginning $S_0$.  From now on we will also assume 
that $S_0=0$.  

A famous problem is to show that when $(S_n)_{n\in\mathbf{N}_0}$ is the simple 
random walk on $\mathbf{Z}^d$ with self-avoiding paths, we have 
$E[|S_n|^2]\approx Cn^{2\nu}$ where
\begin{equation} \label{conj-nu}
\nu=
\begin{cases}
1 & \text{if $d=1$}  \\
3/4 & \text{if $d=2$}  \\
0.588\dots & \text{if $d=3$}  \\
1/2 & \text{if $d\geq4$}
\end{cases}
\end{equation}
and there should be a logarithmic correction for $d=4$.  
The case of $d=1$ is obvious and for $d\geq5$ the
result  has been verified by Hara and Slade \cite{HS92:CMP,HS92:RMP}. 
The other cases are open except for partial results in the $d=4$ case, see
page 400 of \cite{BDGS12} and also \cite{BSTW17} and the references therein.
The same results should hold for weakly 
self-avoiding random walks, that is, when the probability of a path of 
length $n$ is penalized by an exponential term involving the number of 
self-intersections.  

The case $d=1$ is the simplest, but it still presents some 
challenging problems.  For example, consider weakly 
self-avoiding one-dimensional simple random walks $(S_n)_{n\in\mathbf{N}_0}$.  
With $S_0=0$, one could try to characterize the limiting speed,
\begin{equation*}
\lim_{n\to\infty}\frac{1}{n}\left(E\left[S_n^2\right]\right)^{1/2}
\end{equation*}
and there is a fairly complete answer, see Greven and den Hollander 
\cite{GdH93}. Here the key observation was that the occupation 
measure for simple random walk obeys a Markov property similar to 
the Ray-Knight theorem for Brownian motion. There has also been work 
on the continuous-time situation, see van der Hofstad, den Hollander, and
K\"onig \cite{vdHdHK97}.

One limitation of the above models is that they do not take time into 
account.  Of course a real polymer changes its shape over time.
On page 5 of \cite{dH09} den Hollander comments:

\begin{quote}
``We will not (!) consider models where the length or the configuration
of the polymer changes with time (e.g. due to growing or shrinking, 
or to a Metropolis dynamics associated with the Hamiltonian for an 
appropriate choice of allowed transitions). These non-equilibrium 
situations are very interesting and challenging indeed, but so far 
the available mathematics is very thin.''
\end{quote}
The goal of this paper is to make a contribution in this direction, 
for the continuous case and in one dimension.  We would also like to point
out that in our situation, it is not clear that the occupation
measure has a Markov property as in the time-independent case.

\subsection{Setup}

In this section we motivate and define our weakly self-avoiding 
polymer model, and define the radius.  The Rouse model is commonly used to study moving 
polymers without self-avoidance; the polymer is modeled as a sequence 
of balls connected by springs, with friction due to an ambient fluid.  
Doi and Edwards \cite{DE09}, in equation (4.9) on page 92 explain how to take a 
limit and obtain the following stochastic partial differential 
equation (SPDE), which is also called the Edwards-Wilkinson model in 
the context of surface growth (see Funaki \cite{Fun83} for a rigorous derivation),
\begin{equation} \label{eq:edwards-wilkinson}
\begin{aligned}
\partial_tu&=\partial_x^2u+\dot{W}(t,x),  \\
u(0,x)&=u_0(x)   ,
\end{aligned}
\end{equation} 
where $(\dot{W}(t,x))_{t\ge0,x\in[0,J]}$ is a two-parameter white noise.  We
assume that
the intrinsic length of the polymer is $J$, by which we mean that  
$x\in[0,J]$.  Since the ends of the polymer are not fixed, we impose 
Neumann boundary conditions
\begin{equation} \label{boundary} 
\partial_xu(t,0)=\partial_xu(t,J)=0.
\end{equation}
We also assume that $u_0$ is continuous on $[0,J]$.  

As is well-known, we do not expect solutions $(u(t,x))_{t\ge0,x\in[0,J]}$ 
to be differentiable in either variable, so we must regard 
\eqref{eq:edwards-wilkinson} as shorthand for an integral equation, 
usually called the mild form:
\begin{equation} \label{eq:m-form}
u(t,x)=G_t(u_0)(x)+\int_{0}^{t}\int_{0}^{J}G_{t-s}(x,y)W(dyds),
\end{equation}
where 
\begin{equation*}
G_t(f)(x)=\int_{0}^{J}G_t(x,y)f(y)dy 
\end{equation*}
and $G_t(x,y)=G^J_t(x,y)$ is the Neumann heat kernel on $x,y\in[0,J]$ which
solves
\begin{align*}
\partial_t G_t(x,y)&=\partial_x^2G_t(x,y) ,  \\
\partial_xG_t(0,y)&=\partial_xG_t(J,y)=0 , \\
G_0(x,y)&=\delta(x-y).
\end{align*}
Writing 
\begin{equation*}
G^{\mathbf{R}}_t(x)=\frac{1}{\sqrt{4\pi t}}\exp\left(-\frac{x^2}{4t}\right) ,
\end{equation*}
for the heat kernel on $\mathbf{R}$, we have
\begin{equation}
\label{eq:neumann-expansion}
G_t(x,y)=\sum_{k\in\mathbf{Z}}G_t^{\mathbf{R}}(x-y-2kJ)
 +\sum_{k\in\mathbf{Z}}G_t^{\mathbf{R}}(x+y-(2k+1)J).
\end{equation}

It is well-known that for $t>0$ the process $x\to u(t,x)$ is a 
Brownian motion plus a smooth function, see Exercise 3.10 in Chapter 
III.4 of Walsh \cite{wal86} and also Proposition 1 in Mueller and Tribe 
\cite{MT02}.  Therefore we may define an occupation measure and a 
local time as follows, where $m(\cdot)$ is Lebesgue measure,
\begin{equation}  \label{l-time}
\begin{aligned}
L_t(A)&=m\{x\in[0,J]: u(t,x)\in A\}  \\
\ell_t(y)&=\frac{L_t(dy)}{dy}.
\end{aligned}
\end{equation} 
Now we define a weakly self-avoiding process.  For continuous 
processes, the usual way of doing this is to weight the original 
probability measure by the exponential of the integral of local time 
squared, see \cite{dH09} Section 3.1. One might think that 
instead of the fixed-time occupation measure $L_t(A)$ defined above, 
we should consider the overall occupation measure 
$m\{(t,x)\in[0,T]\times[0,J]: u(t,x)\in A\}$.   
However, at different times there is no reason that two parts of the 
polymer cannot be in the same position, so we study $L_t$ as defined 
above.  Then $\ell_t(y)$ represents the density of values of $x$ for 
which $u(t,x)\in dy$, for $t$ fixed.  

If $P_{T,J}$ denotes the 
original probability measure of $(u(t,x))_{t\in[0,T],x\in[0,J]}$, 
we define the probability $Q_{T,J,\beta}$ as follows.  For 
clarity, we will let $E^{P_{T,J}},E^{Q_{T,J,\beta}}$ denote the 
expectations with respect to $P_{T,J}$ and $Q_{T,J,\beta}$ respectively.  
We write $E$ for $E^{P_{T,J}}$.  
Let
\begin{equation} \label{z-func} 
\begin{aligned}
\mathcal{E}_{T,J,\beta}
&=\exp\left(-\beta\int_{0}^{T}\int_{-\infty}^{\infty}\ell_t(y)^2dydt\right),\\
Z_{T,J,\beta}&=E[\mathcal{E}_{T,J,\beta}]
 =E^{P_{T,J}}[\mathcal{E}_{T,J,\beta}].
\end{aligned}
\end{equation} 
where $\beta$ is a parameter representing inverse temperature, 
following the usual convention in statistical mechanics.  
Then we define
\begin{equation} \label{eq:def-Q}
Q_{T,J,\beta}(A)=\frac{1}{Z_{T,J,\beta}}E
 \big[\mathcal{E}_{T,J,\beta}\mathbf{1}_A\big].
\end{equation}
For ease of notation, we will usually drop the subscripts except for 
$T$ and write
\begin{equation*}
P_T=P_{T,J} , \qquad Q_T=Q_{T,J,\beta}  , \qquad 
\mathcal{E}_T=\mathcal{E}_{T,J,\beta} , \qquad Z_T=Z_{T,J,\beta}.
\end{equation*}

Finally, we define the radius of the polymer $u$.  The 
most common definition of the radius of a polymer $(p(x))_{x\in[0,J]}$
involves the end-to-end distance $|p(J)-p(0)|$, but we find it more 
convenient to study the standard deviation of the distance from the center of mass.
Define
\begin{equation}\label{u-bar}
\bar{u}(t)=\frac{1}{J}\int_{0}^{J}u(t,x)dx
\end{equation}
and define the radius of $(u(t,x))_{t\in[0,T],x\in[0,J]}$ to be
\begin{equation*}
R(T,J)=\left[\frac{1}{TJ}\int_{0}^{T}\int_{0}^{J}
 \big(u(t,x)-\bar{u}(t)\big)^2 dxdt \right]^{1/2}.
\end{equation*}
\subsection{Statement of the main result}

For any $\beta, J > 0$ we define 
\begin{equation}\label{eq:h-def}
 h(\beta,J) =\begin{cases}
     1  , &  \text{for }    0\leq  \beta J^{7/2} \leq e,\\ 
             \log(\beta J^{7/2}) ,  & \text{for }   \beta J^{7/2} >e. 
\end{cases} 
\end{equation}
Here is our main result.  
\begin{theorem}
\label{th:main}
There are constants $\varepsilon_0,K_0,K_1>0$ not depending on 
$\beta,J$ such that the following hold.  

(i) For all $J>0$ and $\beta\geq eJ^{-7/2}$ we have
\begin{equation*}
\lim_{T\to\infty}Q_{T}\Big[\varepsilon_0h(\beta,J)^{-1}\beta^{1/3}J^{5/3}
 \leq R(T,J)\leq K_{0} h(\beta,J)^{1/2} \beta^{1/3} J^{5/3}\Big]=1.
\end{equation*}

(ii) For all $J\geq 1$ and $0<\beta< eJ^{-7/2}$ we have
\begin{equation*}
\lim_{T\to\infty}Q_{T}\Big[\varepsilon_0 \beta^{1/3}J^{5/3}
 \leq R(T,J)\leq K_1J^{5/3}\Big]=1.
\end{equation*}
\end{theorem}

\begin{remark} There is a barrier to sharpening the second inequality 
in (ii) with respect to dependence on $\beta$, see Remark 
\ref{rem:ldp} and the explanation after \eqref{req}.  So there is a gap in the upper bound of (ii) regarding $\beta$.  
But we would like to point out that most results for weakly 
self-avoiding polymers do not give end-to-end distance depending on 
$\beta$ (or its analogue) either.  
\end{remark}

\begin{remark}
We give an intuitive justification for Theorem \ref{th:main} in the 
Appendix.  
\end{remark}

\subsection{Outline of the proof}

The following strategy for bounding the right side of 
\eqref{eq:def-Q}  was already used in Bolthausen 
\cite{Bol90}.  

In view of statements (i) and (ii) of Theorem \ref{th:main}, we define
\begin{equation} \label{k2-def}
K_2 =
\begin{cases}
\beta^{1/3}h(\beta,J)^{1/2}K_0 & \text{if $\beta\geq  eJ^{-7/2}$}  \\
K_0 & \text{if $\beta< e J^{-7/2}$}
\end{cases}
\end{equation}
and choose $K_0>0$ later.  Define the following events.  
\begin{equation} \label{events}
\begin{aligned} 
A^{(1)}_T&=\{R(T,J)<\varepsilon_0h(\beta,J)^{-1}\beta^{1/3}J^{5/3}\} ,  \\
A^{(2)}_T&=\{R(T,J)>K_2\beta^{1/3}J^{5/3}\}.   
\end{aligned}
\end{equation}
It suffices to show that for $i=1,2$ we have
\begin{equation*}
\lim_{T\to\infty}Q_{T}\big(A^{(i)}_T\big)=0 .
\end{equation*}
From \eqref{eq:def-Q} we see that it is enough to find:
\begin{enumerate}
\item a lower bound on $Z_T$, derived in Section \ref{sec-part},
\item and an upper bound on 
$E^{P_T}\big[\mathcal{E}_T\mathbf{1}_{A^{(i)}_T}\big]$ 
for $i=1,2$, obtained in Sections \ref{sec-small} and \ref{sec-large}, respectively. 
\end{enumerate}
Finally, the upper bounds divided by the lower bound should vanish 
as $T\to\infty$.  

As mentioned above, Greven and den Hollander \cite{GdH93} give 
a precise result for the length of the growing polymer in the case 
without an extra time parameter.  Their argument depends on special 
properties of the local time or occupation measure, which are not 
available in our case.  Bolthausen's argument \cite{Bol90} starts 
from first principles and gives a less precise result, but does not 
depend on these special properties, so we carry over some of his 
ideas.

\subsection{A conjecture about two dimensions}

We build on the physical reasoning of Flory (see Madras and Slade 
\cite{MS13}, subsection 2.2) and offer a conjecture about the case in 
which the polymer takes values in $\mathbf{R}^2$.  The reasoning is 
given in the Appendix.  

Here we assume that $\mathbf{u}=(u_1,u_2)$ is a 
vector-valued solution to \eqref{eq:edwards-wilkinson}, where 
$\mathbf{u}_0$ is also vector valued, and the Neumann 
boundary conditions \eqref{boundary} still hold.  Also, we assume 
that $\dot{\mathbf{W}}=(\dot{W}_1,\dot{W}_2)$ is a vector of 
independent white noises.  

Since we believe that $x\to\mathbf{u}(t,x)$ behaves like a 
two-dimensional Brownian motion, the local time will not 
exist.  Instead, we should use either
\begin{enumerate}
\item Varadhan's renormalized self-intersection local time \cite{var69}, or 
\item A mollified version of local time such as 
\begin{equation*}
\ell_t^\phi(y) = \int_{0}^{J}\phi(\mathbf{u}(t,x)-y)dx
\end{equation*}
where $y\in\mathbf{R}^2$ and $\phi:\mathbf{R}^2\to[0,\infty)$ is 
compactly supported in a neighborhood of 0.
\end{enumerate}
Replacing $\ell$ by one of these alternatives, we define $R(T,J)$ 
as before.  We do not state our conjecture as precisely as Theorem 
\ref{th:main}, nor do we speculate about the dependence of $R(T,J)$ 
on $\beta$.  

\begin{conj}
\label{conj:dim-2}
With high probability,
\begin{equation*}
R(T,J)\approx J^{5/4}.
\end{equation*}
\end{conj}

\section{ Lower Bound on the Partition Function}  \label{sec-part} 

In this section we derive a lower bound on the partition function 
$Z_T$ which was defined in \eqref{z-func}.  This bound is 
given in the following proposition.  Recall that $h$ was defined in 
\eqref{eq:h-def}. 
\begin{prop} \label{prop-part}  
For any $\beta>0$ and $J>0$ we have 
\begin{equation*}
\liminf_{T\rightarrow \infty } \frac{1}{T}\log Z_T   \geq  -CJ^{1/3}h(\beta,J) \beta^{2/3}, 
\end{equation*}
where $C>0$ is a constant independent from $J$ and $\beta$. 
\end{prop} 
The proof of Proposition \ref{prop-part} is delayed to the end of this section as we will need additional ingredients in order to derive this result. 
 
We first define the radius $R_u$ and then derive the scaling properties of 
$R_u(T,J)$, $\mathcal{E}_{T,J,\beta}$ and $Z_{T,J,\beta}$ in $J$. 
We therefore  introduce some additional definitions and notation, which will be used later on. Let
\begin{equation} \label{theta}
\theta_{u}(t,J)
:=\left[\frac{1}{J}\int_{0}^{J}\big(u(t,x)-\bar{u}(t)\big)^2dx\right]^{1/2},
\quad
0\leq t \leq T.
\end{equation}
We also write $\theta_\varphi(J)$ when $\varphi$ is a function which may not 
depend on $t$; the definition is the same as for $\theta_u$.  

Define
\begin{equation} \label{eq:H-def} 
R_\varphi(T,J)=\left(\frac{1}{T}
  \int_{0}^{T}\theta_\varphi(t,J)^2dt\right)^{1/2}.
\end{equation}

Let  $(w(t,x))_{t\geq 0,x\in D}$ be a double-indexed 
stochastic process, where $D \subset \mathbb{R}$ is a compact set.  
In analogy to the local time of $u$ which was defined in 
\eqref{l-time}, we define 
$\ell_{t}^{w}= (\ell^{w}_{t}(y))_{y\in \mathbf{R}}$ as the local 
time of $w(t,\cdot)$, whenever it exists. Moreover we define 
$\mathcal{E}^{w}_{T,J,\beta}$ ($Z^{w}_{T,J,\beta}$) as exponential of 
the squared local time of $w$ (partition function), which corresponds 
to $\mathcal{E}_{T,J,\beta}$ ($Z_{T,J,\beta}$) in \eqref{z-func}. 
Recall that $Q_{T,J,\beta}$ was defined in \eqref{eq:def-Q}. In 
similar way we define $Q^{w}_{T,J,\beta}$ when we refer to the 
process $w$. 
\\ 

Finally, recall that $u$ satisfies \eqref{eq:edwards-wilkinson} on $x\in[0,J]$, with boundary conditions \eqref{boundary}. 
In the following lemma we establish the scaling properties of $R_u(T,J)$, $\mathcal{E}^{u}_{T,J,\beta}$ and $Q^{u}_{T,J,\beta }(\cdot)$ in $J$. 
\begin{lemma} \label{lem-scale}
Let $T,J>0$. Define
 \begin{equation} \label{def-v}
v(t,x):=J^{-1/2}u(J^2t,Jx).
\end{equation}
Then the following holds: 
\begin{itemize} 
\item[\textbf{(i)}] $v$ satisfies \eqref{eq:edwards-wilkinson} on $x\in[0,1]$, with a different white 
noise.  
\item[\textbf{(ii)}] For any constant $\beta>0$ we have  
\begin{equation*}
\mathcal{E}^{u}_{T,J,\beta }
\stackrel{\mathcal{D}}{=} \mathcal{E}^{v}_{TJ^{-2},1,\beta J^{7/2} },  \quad Q^{u}_{T,J,\beta }(\cdot) =Q^{v}_{TJ^{-2},1,\beta J^{7/2} }(\cdot).
\end{equation*}
\item[\textbf{(iii)}] \begin{equation*}
R_u(T,J) \stackrel{\mathcal{D}}{=} J^{1/2}R_v(J^{-2}T,1). 
\end{equation*}
\end{itemize} 
\end{lemma} 
The proof of Lemma \ref{lem-scale} is postponed to Section \ref{section-scale}. 

From Lemma \ref{lem-scale}(ii) it follows that we can prove Proposition 
\ref{prop-part} for $J=1$ and then use the scaling properties of $Z_{T}$ to 
generalize the result for any $J>0$. Therefore in the remainder of the 
section we assume that $J=1$. 
\medskip \\
 
Next we state a few useful facts on the Fourier decomposition of the solution 
to \eqref{eq:edwards-wilkinson}, which are taken from Chapter III.4 of 
\cite{wal86}. Note that the stochastic heat equation in Chapter III.4 of 
\cite{wal86} includes also a linear drift term, however this only changes the 
eigenvalues of the equation and does not affect the eigenfunctions. 
 
\medskip 
Let $(\varphi_n(\cdot),-\lambda_n)_{n\geq0}$ be the sequence of 
orthonormal Neumann eigenfunctions and eigenvalues of the Laplacian on 
$[0,1]$, with $(\lambda_n)_{n\ge0}$ in increasing order. Then
\begin{equation} \label{phi-0}
\varphi_0(x)=1, \qquad \lambda_0=0
\end{equation}
and for $n\geq1$,
\begin{equation} \label{phi-n}
\varphi_n(x)=\sqrt{2}
 \cos(n\pi x), \qquad \lambda_n=n^2\pi^2.
\end{equation}
Recall that the mild form of \eqref{eq:edwards-wilkinson} was
defined in \eqref{eq:m-form}. Define 
\begin{equation} \label{N-def} 
N(t,x)= \int_{0}^{t}\int_0^JG_{t-s}(x,y)W(dyds), \quad t\geq 0, \ x\in [0,J]. 
\end{equation} 
It follows that for $J=1$,
\begin{equation} \label{n-decom} 
N(t,x)=\sum_{n=0}^{\infty}X^{(n)}_t\varphi_n(x),
\end{equation}
where
\begin{equation}  \label{ou-n}
\begin{aligned} 
dX^{(n)}_t&=-\lambda_nX^{(n)}_tdt+dB^{(n)}_t , \\
X^{(n)}_0&=0
\end{aligned}
\end{equation} 
and $(B^{(n)}_\cdot)_{n\geq0}$ is a collection of independent Brownian 
motions.

\medskip 

As in Section 3 of \cite{Bol90} we define a measure 
$\hat P_T^{(a)}=\hat{P}^{(a)}_{T,1,\beta}$ that adds a drift to the process.  In our 
situation, it would be logical to simply add a constant drift to the white 
noise.  However, adding such drift would just shift the solution 
to \eqref{eq:edwards-wilkinson} to the right (or left).  So we would 
like to add a drift which increases or decreases with $x$.  

In order to do that, we take the noise $\dot{W}$ and add a 
drift $a\varphi_{1}(\cdot)$. In what what follows we fix $T>0$. Recall that 
\begin{equation} 
\label{phi-norm} 
\int_{0}^{1}\varphi^{2}_{1}(x)dx =1. 
\end{equation} 
Then the Cameron-Martin formula for spacetime white noise (see Dawson  
\cite{Daw78}, Theorem 5.1) gives 
\begin{equation} \label{q-hat}
\begin{aligned} 
\frac{d\hat P_T^{(a)}}{dP_T}
&=\exp\left(\int_{0}^{T}\int_{0}^{1}a\varphi_{1}(x)W(dxdt)
 -\frac{1}{2}\int_{0}^{T}\int_{0}^{1}a^2\varphi_{1}^{2}(x)dxdt\right)  
   \\
&=\exp\left(\int_{0}^{T}\int_{0}^{1}a\varphi_{1}(x)W(dxdt)\right)
 \Big/\hat{\zeta}(T,a),
\end{aligned}
\end{equation} 
where
\begin{equation} 
\label{zeta-hat}
\hat{\zeta}(T,a)
 =\exp\left(\frac{Ta^2}{2}\right).
\end{equation}

We write $\hat{E}$ for the expectation $E^{\hat P_T^{(a)}}$ with respect to 
$\hat P_T^{(a)}$.  Let $\beta >0$, and recall that $\ell_t(\cdot)$ is the 
local time of the process $x\to u(t,x)$.  Now using Jensen's inequality, we get
\begin{equation} \label{z-beta-low} 
\begin{aligned}
 \log Z_T  
&= \log \hat{E}\left[\exp\left(-\beta\int_{0}^{T}
 \int_{-\infty}^{\infty}\ell_t(y)^2dydt 
  -\log\frac{d\hat P_T^{(a)}}{dP_{T}}\right)\right]   \\
&\geq -\beta\hat{E}\left [\int_{0}^{T}
 \int_{-\infty}^{\infty}\ell_t(y)^2dydt \right ]
 - \hat{E}\left[\log\frac{d\hat P_T^{(a)}}{dP_{T}}\right]  \\
&=: -\beta I_{1}(T) - I_{2}(T).
\end{aligned}
\end{equation}
In what follows we derive the asymptotic behaviour of 
$I_{i}(T)$ for $i=1,2$ when $T\rightarrow \infty$, which will 
help us to prove Proposition \ref{prop-part}. 

\begin{prop} \label{lemma-i-1}
There exist positive constants $C_{1},C_{2}>0$ not depending on $J,T,a$
such that:
\begin{itemize} 
\item[\textbf{(i)}] for any $0<a \leq 1$ we have 
\begin{equation*}
\lim_{T\rightarrow \infty }\frac{I_1(T)}{T}
 =\frac{C_{1}}{a},
\end{equation*}
\item[\textbf{(ii)}]  for any $a>1$ we have 
\begin{equation*}
\lim_{T\rightarrow \infty }\frac{I_1(T)}{T}
 =C_2 \frac{ \log a}{a}.
\end{equation*}

\end{itemize} 
\end{prop} 
\begin{proof} 
From \eqref{eq:m-form} it follows that for any continuous initial condition $u_0$, as $t\to\infty$,  $G_t(u_0)(x)$ converges to a constant uniformly in $x$. Since adding a constant does not change the radius of the solution, we can ignore the influence of the initial data and start with $u(0,\cdot)$ which is a pinned stationary distribution with respect to $\hat{P}$ (see e.g. equation (2.1) in \cite{MT02}).
Let $t_{0} \geq 0$ and $x_{0}\in [0,1]$. For $ t_0\leq t$ and 
$x\in[0,1]$ define 
\begin{equation} \label{eq:u-pinned}
\begin{aligned} 
u_{t_0,x_0}(t,x)&:=\int_{t_0}^{t}\int_{0}^{1}G_{t-r}(x,y)\widetilde {W}(dydr)  \\
&\quad+\int_{-\infty}^{t_0}\int_{0}^{1}\big[G_{t-r}(x,y)-G_{t_0-r}(x_0,y)\big]\widetilde{W}(dydr) \\
&\quad+a\int_{-\infty}^{0}\int_{0}^{1}\big[G_{t-r}(x,y)-G_{t_0-r}(x_0,y)\big]\varphi_1(y)dydr,
\end{aligned}
\end{equation}
where $\widetilde{W}$ is a space-time white noise.

We first study the solution $u$ to \eqref{eq:u-pinned} 
under a $\hat P_T^{(a)}$, which adds 
drift $a \varphi_{1}(y)$ to the noise $\widetilde{W}(dyds)$ for $s\in[0,T]$.

Therefore, we replace $\widetilde W(dyds)$ with 
$\hat{W}(dyds)+a\varphi_{1}(y)dyds$ for $s\in[0,T]$, $y\in[0,1]$. So from \eqref{eq:m-form} we get that $u_{t_0,x_0}$ satisfies  the following equation under 
$\hat P_T^{(a)}$, for $t_0\leq t \leq T$ and $x\in[0,1]$, 
\begin{equation} \label{eq:u-drift}
\begin{aligned} 
u_{t_0,x_0}(t,x)&:=\int_{t_0}^{t}\int_{0}^{1}G_{t-r}(x,y)\hat{W}(dydr)  \\
&\quad +a\int_{t_{0}}^{t}\int_{0}^{1}G_{t-r}(x,y)\varphi_1(y)dydr  \\
&\quad+\int_{-\infty}^{t_0}\int_{0}^{1}\big[G_{t-r}(x,y)-G_{t_0-r}(x_0,y)\big]\hat{W}(dydr)   \\
&\quad+a\int_{-\infty}^{t_{0}}\int_{0}^{1}\big[G_{t-r}(x,y)-G_{t_0-r}(x_0,y)\big]\varphi_1(y)dydr  \\
&=: (A)+a\cdot(B)+(C)+a\cdot(D).
\end{aligned}
\end{equation} 
Recall that $(\varphi_{n}, \lambda_{n})$ were defined in \eqref{phi-n}. We will use the $\mathbf{L}^2$-eigenfunction expansion of the heat kernel 
\begin{equation} \label{g-exp}
G_t(x,y)=\sum_{n=0}^{\infty}e^{-\lambda_nt}\varphi_n(x)\varphi_n(y).
\end{equation}
Since $\lambda_{n}$ is the eigenvalue corresponding to $\varphi_{n}$, we have for every $n \geq 1$,
\begin{equation} \label{hk-int}
\int_{0}^{1}\varphi_{n}(y)G_{t-s}(x,y)dy=e^{-\lambda_{n} (t-s)}\varphi_{n}(x)
 =\exp\left(- \pi^2 n^{2}(t-s)\right)\varphi_{n}(y).
\end{equation}
 
In the following lemma we verify that the integrals (C) and (D) converge. 
Note that (C) and (D) depend on $(t,t_{0},x,x_{0})$ so we often abbreviate 
$(C)(t,t_{0},x,x_{0})$ and $(D)(t,t_{0},x,x_{0})$. 

\begin{lemma} \label{lem-conCD} For any $T\geq 0$ the following holds:
\begin{itemize}  
 \item[\textbf{(i)}] 
\begin{equation*}
 \sup_{t_0\leq t \leq T}\sup_{x,x_{0}\in[0,1]} E[(C)
 (t,t_{0},x,x_{0})^{2}] < \infty, 
\end{equation*}
 \item[\textbf{(ii)}] 
\begin{equation*}
 \sup_{t_0\leq t \leq T}\sup_{x,x_{0}\in[0,1]}  |(D)(t,t_{0},x,x_{0})| <\infty .
\end{equation*}
\end{itemize} 
\end{lemma} 
The proof of Lemma \ref{lem-conCD} is postponed to Section \ref{pf-prop-string}. 
\\

In the following two lemmas, which will be proved in Section \ref{pf-prop-string}, we derive two essential properties of the pinned string. 

Assume now that $t_{0}$ and $x_{0}$ are fixed. For any $t>0$, let $\hat{g}_{t, x_1,x_2}$ be the density 
function for $u_{t_{0},x_{0}}(t,x_2)-u_{t_{0},x_{0}}(t,x_1)$ with respect to $\hat P_T^{(a)}$.
First, we reformulate the expected local time integral in terms of an integral over $\hat{g}_{t,x_1,x_2}(0)$.
 \begin{lemma} [Rephrasing the local time integral] \label{l-t-ident}
 For any $0\leq t\leq T$ we have 
\begin{equation} \label{eq:reduction-to-difference}
\hat{E}\left[\int_{-\infty}^{\infty}\ell_t(y)^2dy  \right]
= \int_{0}^{1}\int_{0}^{1}\hat{g}_{t, x_1,x_2}(0)dx_2dx_1.
\end{equation}
\end{lemma} 
Next we derive the shift invariance property of the pinned string.
\begin{lemma} [Shift invariance of the pinned string] \label{lemma-shift} 
Let $0\leq t_{0} \leq T$ and $x_{0}\in[0,1]$. Then under the measure $\hat P_T^{(a)}$ the random field 
\begin{equation*}
U_{t_{0},x_{0}}(t,x_{1},x_{2}) := u_{t_0,x_0}(t,x_{1})-u_{t_0,x_0}(t,x_2), \quad  t\in [0,T],\, x_1,x_{2}\in[0,1], 
\end{equation*}
is stationary in $t$. That is, for any $t_{0}<t_1<t_2 \leq T$ we have
\begin{align*}
\big(U_{t_{0},x_{0}}(t_{1},x_{1},x_{2})\big)_{x_1,x_{2}\in[0,1]} 
\stackrel{\mathcal{D}}{=} 
\big(U_{t_{0},x_{0}}(t_{2},x_{1},x_{2})\big)_{x_1,x_{2}\in[0,1]}.
\end{align*}
\end{lemma} 

From Lemmas \ref{l-t-ident} and \ref{lemma-shift}, it follows that in order to bound $I_1(T)$, we can restrict our discussion to the case where $t=t_0=0$ and $x_0=1/2$ in \eqref{eq:u-drift}. Then we have
\begin{equation} \label{pinned-0}
\begin{aligned}
u_{0,1/2}(0,x)
&= \int_{0}^{\infty}\int_{0}^{1}\big[G_{r}(x,y)-G_{r}(1/2,y)\big]\hat{W}(dydr)   \\
&\quad+a\int_{0}^{\infty}\int_{0}^{1}\big[G_{r}(x,y)-G_{r}(1/2,y)\big]\varphi_1(y)dydr  \\
&=: (E)(x)+a\cdot(F)(x).
\end{aligned}
\end{equation} 
Recall that $\{\ell_{t}(y)\}_{y\in\mathbf{R}}$ is the local time of 
$u_{0,1/2}(t,\cdot)$. From \eqref{z-beta-low} and Lemmas \ref{l-t-ident} and 
\ref{lemma-shift}, we therefore have  
\begin{equation} \label{i-1-g}
I_1(T)=T\hat{E}\left[\int_{-\infty}^{\infty}\ell_0(y)^2dy\right]= T \int_{0}^{1}\int_{0}^{1}\hat{g}_{0, x_1,x_2}(0)dx_2dx_1.
\end{equation}
From \eqref{i-1-g} we conclude that the scaling properties of $I_1(T)$ are given by $\hat{g}_{0, x_1,x_2}$.
 
\textbf{Notation.} 
In order to simplify the notation we write $\hat{g}_{x_1,x_2}$ instead of $\hat{g}_{0, x_1,x_2}$ for the rest of this section. We also use the notation $u(t,x)$ instead of $u_{0,1/2}(t,x)$ where there is no ambiguity. 

In the following lemma we derive some essential bounds on the second moment for the increments of $u_{0,1/2}$, in the drift-less case, that is when $a=0$ in \eqref{pinned-0}. Since $u_{0,1/2}$ is a Gaussian process, this bound also applies  to the variance of its increments for any $a>0$. 
\\

\begin{lemma} \label{lemma:2nd-mom}
Assume that  $a=0$ in \eqref{pinned-0}. Then there exist constants $C_{1},C_{2}  >0$ such that for all $x_{1},x_{2} \in [0,1]$ and $t\in [0,T]$, 
\begin{equation*}
 C_{1}|x_1-x_2|  \leq   E\left[\left(u(t,x_1)-u(t,x_2)\right)^2\right] \leq C_{2}|x_1-x_2|.
 \end{equation*}
\end{lemma} 
The proof of Lemma \ref{lemma:2nd-mom} is given in Section \ref{pf-int-conv}.

In order to study $\hat{g}_{x_1,x_2}$, we need analyze the drift term (F) in \eqref{pinned-0}. Since $\varphi_1(1/2)=0$ we get from \eqref{hk-int} that  
\begin{align*}
(F)(x) &= \int_{0}^{\infty}e^{-\lambda_1t}\left[\varphi_1(x)-\varphi_1(1/2)\right]dt \\
&= \int_{0}^{\infty}e^{-\lambda_1t}\varphi_1(x)dt   \\
&=\frac{\sqrt{2}}{\pi^{2}}\cos(\pi x),
\end{align*}
where we used \eqref{phi-n} in the last equality. 

Define 
$$
\mathcal D(x_{1},x_{2}):=(F)(x_{1}) -(F)(x_{2}), \quad  0\leq x_{1}, x_{2} \leq 1. 
$$
In the following lemma we derive a lower bound on $\mathcal D(x_{1},x_{2})$.  
\begin{lemma} \label{lem-drift} 
For all $0\leq x_{1} \leq x_{2} \leq 1$ we have 
$$
 \mathcal D(x_{1},x_{2}) \geq 
 \begin{cases}
           \frac{\sqrt{2}}{8\pi^{2}} (x_{2} -x_{1})(x_{2} +x_{1}), & \text{for }   x_{1}+x_{2} \leq 1,\\ 
         \frac{\sqrt{2}}{2\pi^{2}} (x_{2} -x_{1})\left(1-\frac{x_{2}+x_{1}}{2}\right), & \text{for }   1< x_{1}+x_{2} \leq 2. 
\end{cases}
$$
\end{lemma}
The proof of Lemma \ref{lem-drift} is given in Section \ref{pf-int-conv}. 

We define ${g}_{ x_1,x_2}$ as the density function for $u(0,x_2)-u(0,x_1)$ 
in \eqref{pinned-0} when $a=0$. From Lemmas \ref{lemma:2nd-mom} and 
\ref{lem-drift}  we get that there exist constants $C_{1},C_{2} >0$ such that 
\begin{equation} \label{g-bound} 
\begin{aligned} 
&\int_{0}^{1}\int_{0}^{1}\hat{g}_{x_1,x_2}(0)dx_2dx_1 
= 2\int_{0}^{1}\int_{x_1}^{1}\hat{g}_{x_1,x_2}(0)dx_2dx_1  \\
&= 2\int_{0}^{1}\int_{x_1}^{1}g_{x_1,x_2}( \mathcal D(x_{1},x_{2}) )dx_2dx_1  \\
&\leq C_{1}\int_{0}^{1}\int_{x_{1}}^{1}(x_{2}-x_{1})^{-1/2}
 \exp\left(-\frac{C_{2}a^2  \mathcal D(x_{1},x_{2})^{2}}{x_{2}-x_{1}}\right)dx_{2}dx_{1}  \\
&\leq C_{1}\int_{0}^{1}\int_{x_{1}}^{1}\mathbf{1}_{\{x_{1}+x_{2} \leq 1\}} (x_{2}-x_{1})^{-1/2} \\
&\hspace{2cm}  \times\exp\left(-C_{3}a^2(x_{1}+x_{2})^2(x_{2}-x_{1})\right)dx_{1}dx_{2}\\
&\quad +C_{1}\int_{0}^{1}\int_{x_{1}}^{1} \mathbf{1}_{\{x_{1}+x_{2}> 1\}} (x_{2}-x_{1})^{-1/2} \\
 &\hspace{2cm} \times
 \exp\left(-C_{4}a^2\left(1- \frac{x_{1}+x_{2}}{2}\right)^2(x_{2}-x_{1})\right)dx_{1}dx_{2}\\
 &=: C_{1}\big(\mathcal I_{1}(a)+ \mathcal I_{2}(a)\big).
\end{aligned}
\end{equation} 
The result then follows from following lemma, which is proved in Section \ref{pf-int-conv}.
 \begin{lemma}  \label{int-conv}
 There exist constants $C_{1},C_{2}>0$ not depending on $a$, such 
that for $i=1,2$ we have: 
 \begin{itemize} 
 \item[\textbf{(i)}] for all $0<a\leq 1$,
 $$\mathcal I_{i}(a) \leq \frac{C_{1}}{a},$$ 
 \item[\textbf{(ii)}] for all $a>1$, 
 $$\mathcal I_{i}(a) \leq C_{2}\frac{\log a}{a}.$$
 \end{itemize} 
 \end{lemma} 
\end{proof}

Next we analyze $I_{2}(T)$ from (\ref{z-beta-low}). 
\begin{lemma}  \label{lemma-i-2}
For any $J,T,a > 0$ we have 
\begin{equation*}
I_{2}(T)  = \frac{a^{2}T}{2}. 
\end{equation*}
\end{lemma} 

\begin{proof} 
From \eqref{zeta-hat} and \eqref{z-beta-low} it follows that
\begin{equation} \label{c-der} 
\begin{aligned}
I_{2}(T) &= \hat{E}\left[\log \frac{d\hat P_T^{(a)}}{dP_T}\right] \\
&=\hat{E}\left[\int_{0}^{T}\int_{0}^{J}a\varphi_{1}(x)W(dxdt)\right]  
   -\log\hat{\zeta}(T,a) \\
&=\hat{E}\left[W(T,a \varphi_1)\right] 
   -\frac{a^2T}{2},
\end{aligned}
\end{equation} 
where we write
\begin{equation*}
W(T,a \varphi_1)=\int_{0}^{T}\int_{0}^{J}a\varphi_{1}(x)W(dxdt)
\end{equation*}
Using \eqref{q-hat} we get 
\begin{equation} \label{i-2-1}
\begin{aligned}
\hat{E}\left[W(T,a\varphi_1)\right]  
&=E\left[W(T,a\varphi_1)\frac{d\hat P_T^{(a)}}{dP_T}\right]   \\
&=\frac{a}{\hat{\zeta}(T,a)}
 E\left[W(T,\varphi_1)\exp\left(aW(T,\varphi_1)\right)\right]    \\
&=\frac{a}{\hat{\zeta}(T,a)} \frac{d}{da}
 E\left[ \exp\left(aW(T,\varphi_1)\right)\right].
\end{aligned}
\end{equation} 

In the preceding line we have the derivative of a moment 
generating function.  Let
\begin{equation*}
X= W(T,\varphi_1),
\end{equation*}
and
\begin{equation*}
\psi(a)=E\left[ \exp\left(aX\right)\right]. 
\end{equation*}
Note that $X\sim N(0,\sigma^2)$ and from \eqref{phi-norm} we conclude that
\begin{equation*}
\sigma^2=\int_{0}^{T}\int_{0}^{1}\varphi_{1}(y)^2dy = T.
\end{equation*}
It follows that 
\begin{align*}
\psi(a)&=\exp\left(\frac{a^2T}{2}\right), \\
\frac{d}{da}\psi(a)&=aT\exp\left(\frac{a^2T}{2}\right). 
\end{align*}
Using this moment generating function computation in \eqref{i-2-1} 
and using \eqref{zeta-hat}, we get
\begin{equation} \label{eq:W}
 \hat{E}\left[W(T,a\varphi_1)\right]
= \frac{1}{\hat{\zeta}(T,a)} a^2T
 \exp\left(\frac{a^2T}{2}\right)
=a^2T.  
\end{equation}
Pulling together \eqref{c-der} and \eqref{eq:W}, we have
\begin{equation*}
I_{2}(T)  = \hat{E}\left[W(T,a\varphi_1)\right]
     -\log\hat{\zeta}(T,a)
= a^{2}T - \frac{a^2}{2}T
= \frac{1}{2}a^{2}T.
\end{equation*}
\end{proof} 

Now we are ready to prove Proposition \ref{prop-part}.

\begin{proof} [Proof of Proposition \ref{prop-part}]

We first prove the proposition for $J=1$, that is for a solution $v$ to \eqref{eq:edwards-wilkinson} on $x\in[0,1]$.

From \eqref{z-beta-low}, Proposition \ref{lemma-i-1}(i), and Lemma \ref{lemma-i-2} we get for any $0\leq \beta \leq e$ and $0\leq a\leq 1$, 
\begin{align*}
\liminf_{T\rightarrow \infty}\frac{1}{T}\log Z_{T,1,\beta}  
&\geq \lim_{T\rightarrow \infty} \frac{1}{T} \big(-\beta I_{1}(T) - I_{2}(T) \big) \\ 
&= -C\frac{\beta}{a}- \frac{a^2}{2}.
\end{align*}
Next we choose $a=(\beta/e)^{1/3}$, and get that there exists a constant $\tilde C>0$ such that 
\begin{equation} \label{z1-lim}
\liminf_{T\rightarrow \infty} \frac{1}{T}\log Z_{T,1,\beta}  
\geq - \tilde C\beta^{2/3} . 
 \end{equation}
 
Again from \eqref{z-beta-low}, Proposition \ref{lemma-i-1}(ii) and Lemma \ref{lemma-i-2} we get for any $\beta > e$ and $a>1$, 
\begin{align*}
\liminf_{T\rightarrow \infty}\frac{1}{T}\log Z_{T,1,\beta}  
&\geq \lim_{T\rightarrow \infty} \frac{1}{T} \big(-\beta I_{1}(T) - I_{2}(T) \big) \\ 
&= -C\frac{\beta \log a}{a}- \frac{a^{2}}{2}.
\end{align*}
As before we choose $a=(\beta/e)^{1/3}$, and get that there exists a constant $\hat C>0$ such that 
\begin{equation} \label{z1-lim-log}
\liminf_{T\rightarrow \infty} \frac{1}{T}\log Z_{T,1,\beta}  
\geq - \hat  C\beta^{2/3} \log \beta. 
 \end{equation}

Let $J>0$ and let $u$ be the solution to \eqref{eq:edwards-wilkinson} on $x\in[0,J]$. Now we use Lemma \ref{lem-scale}(ii) and \eqref{z-func} to get  $Z^{u}_{T,J,\beta } =Z^{v}_{TJ^{-2},1,\beta J^{7/2} }$. Together with \eqref{z1-lim} we have for $0\leq \beta J^{7/2}  \leq e $,
 \begin{equation*}
 \begin{aligned} 
\liminf_{T\rightarrow \infty} \frac{1}{T}\log  Z^{u}_{T,J,\beta } &=J^{-2}  \liminf_{T\rightarrow \infty} \frac{1}{TJ^{-2}}\log Z^{v}_{TJ^{-2},1,\beta J^{7/2}}\\
&\geq - J^{-2} \beta^{2/3} J^{7/3}C\\
&\geq -   \beta^{2/3} J^{1/3}  C.
\end{aligned} 
 \end{equation*}

Similarly using \eqref{z1-lim-log} we get for all $\beta J^{7/2} \geq e$, 
 \begin{equation*}
 \begin{aligned} 
\liminf_{T\rightarrow \infty} \frac{1}{T}\log  Z^{u}_{T,J,\beta } &=J^{-2}  \liminf_{T\rightarrow \infty} \frac{1}{TJ^{-2}}\log Z^{v}_{TJ^{-2},1,\beta J^{7/2}}\\
&\geq - CJ^{-2} \beta^{2/3} J^{7/3}\log (\beta J^{2/7}).
\end{aligned} 
 \end{equation*}
By the definition of $h(\beta, J)$ from \eqref{eq:h-def}, the result follows. 
\end{proof}

\section{Small distance tail estimate} \label{sec-small}
In this section we derive an upper bound for the tail behaviour of $R(T,J)$,
which is given in the following proposition. Recall that $Q_{T}$ and 
$A^{(1)}_{T}$ were defined in \eqref{eq:def-Q} and \eqref{events},
respectively, and that $A^{(1)}_{T}$ depends on $\varepsilon_0$.
\begin{prop}  \label{prop-small}
We can choose $\varepsilon_0>0$ small enough so that
\begin{equation*}
\lim_{T\to\infty}Q_{T}\big(A^{(1)}_{T}\big)=0.
\end{equation*}
\end{prop}
\begin{proof} 
We first prove the proposition for $J=1$, that is, for a solution $v$ to \eqref{eq:edwards-wilkinson} on $x\in[0,1]$.

We first define the event $A_{K,T} = \{R(T,1) \leq K\}$. Recall that $\theta_{v}$ was defined in \eqref{theta}.  From \eqref{eq:H-def} we get  
\begin{equation*}
R(T,1)^{2} = \frac{1}{T}\int_{0}^{T} \theta_{v}(t,1)^{2}dt.  
\end{equation*}
Hence we have on $A_{K,T}$ 
\begin{equation} \label{kj1} 
|\{t\in[0,T]: \, \theta_{v}(t,1)^{2}  \leq 2K^{2} \} | \geq T/2, 
\end{equation} 
where $|\cdot|$ denotes the Lebesgue measure. From \eqref{theta} it follows that if $ \theta_{v}(t,1)^{2}  \leq 2K^{2}$ then 
\begin{equation} \label{kj2}  
| \{x\in[0,1] : v(t,x) \in [\bar v(t) - 2K , \bar v(t) + 2K] \}| \geq 1/2. 
\end{equation} 
Let $d_{t}^{\pm} = \bar v(t) \pm2 K$, and note that $d_{t}^{+} - d_{t}^{-} = 4K$.  Then from \eqref{l-time}, \eqref{kj1} and \eqref{kj2} it follows that on $A_{K,T}$ we have 
\begin{equation}  \label{kj3}
\left|\left\{t\in[0,T]: \, \int_{d^{-}_{k}}^{d_{k}^{+} } \ell^{v}_{t}(x) dx
\geq 1/2 \right\}\right | \geq T/2.
\end{equation} 
Using Jensen's inequality and \eqref{kj3}, we get on $A_{T,K}$
\begin{equation}  \label{l-low}
\begin{aligned}
\int_{0}^{T}\int_{-\infty}^{\infty}\ell^{v}_t(y)^2dy dt
&\geq 4K \int_{-\infty}^{\infty}\left(\int_{d^{-}_{K}}^{d^{+}_{K}}\ell_t(y)^2\frac{dy}{4 K} \right)dt \\
&\geq 4 K \int_{0}^{T}
 \left(\int_{d^{-}_{K}}^{d^{+}_{K}}\ell_t(y)\frac{dy}{4 K}\right)^2 dt \\
&\geq \frac{T}{32 K}.
\end{aligned}
\end{equation} 
Note that on the event $A^{(1)}_{T}$ in \eqref{events} we have 
\begin{equation}\label{rec-a-1}
R(T,1) <\varepsilon_{0}h(\beta,1)^{-1} \beta^{1/3}. 
\end{equation}
Hence from \eqref{l-low} we get that on $A^{(1)}_{T}$, 
\begin{align*}
\int_{0}^{T}\int_{-\infty}^{\infty}\ell_t(x)^2dxdt \geq \frac{T}{32\varepsilon_{0}h(\beta,1)^{-1}\beta^{1/3}}.
\end{align*}

Recall the definition of $\mathcal{E}_{T}$ in \eqref{z-func}. 
Thus for any $\varepsilon_{0}>0$ we have 
\begin{equation} \label{p-small}
E\Big[\mathcal{E}_{T}\mathbf{1}\big(R(T,1)<\varepsilon_{0}h(\beta,1)^{-1}\beta^{1/3}\big)\Big] 
\leq \exp\left(-\beta \frac{T}{32\varepsilon_{0}h(\beta,1)^{-1}\beta^
{1/3}}\right). 
\end{equation}

Using Proposition \ref{prop-part} with $J=1$ and \eqref{p-small} we get 
\begin{equation} \label{smal-fin}
\begin{aligned} 
\lim_{T\rightarrow \infty}  \frac{1}{T}&\log 
 Q_{T}\big(R(T,1)<\varepsilon_{0}h(\beta,1)^{-1}\beta^{1/3}\big)  \\
& \leq \lim_{T\rightarrow \infty} \frac{1}{T} \log E
 \Big[\mathcal{E}_{T}
 \mathbf{1}\left(R(T,1)<\varepsilon_{0}h(\beta,1)^{-1}\beta^{1/3}\right)\Big] \\
&\qquad -\liminf_{T\rightarrow \infty}\frac{1}{T}\log Z_{T} \\
&\leq  -\beta^{2/3}h(\beta,1)\left(\frac{1}{32\varepsilon_{0}} - C\right). 
\end{aligned} 
\end{equation}
Hence by choosing $\varepsilon_{0}$ small enough we get the result for $J=1$. 

Let $J>0$ and let $u$ be the solution to \eqref{eq:edwards-wilkinson} on $x\in[0,J]$. Define 
\begin{equation*}
\beta_{u}=J^{-7/2}\beta_v. 
\end{equation*}
Now we use Lemma \ref{lem-scale}(ii) and (iii) together with \eqref{smal-fin} and 
\begin{equation} \label{h-scale} 
h(\beta_{u},J) = h(\beta_{v},1), 
\end{equation} 
to get 
\begin{equation}  \label{scale_trans}
\begin{aligned} 
&\frac{1}{T}Q_{T,J,\beta_{u}}\big( R_u(T,J)\leq \varepsilon_0h(\beta_{u},J)^{-1}\beta_u^{1/3}J^{5/3}\big) \\
&=\frac{1}{T}Q_{TJ^{-2},1,\beta_{v}}\big( J^{1/2}R_v(TJ^{-2},1)\leq\varepsilon_0h(\beta_{v},1)^{-1} J^{-7/6}\beta_v^{1/3}J^{5/3}\big) \\
&=J^{-2}\frac{1}{TJ^{-2}} Q_{TJ^{-2},1,\beta_{v}}\big(  R_v(TJ^{-2},1)\leq \varepsilon_0h(\beta_{v},1)^{-1} \beta_{v}^{1/3}\big).
\end{aligned}
\end{equation}   
It follows that 
\begin{align*} 
\lim_{T\rightarrow \infty }\frac{1}{T}&Q_{T,J,\beta_{u}}\big( R_u(T,J)\leq \varepsilon_0h(\beta_{u},J)^{-1} \beta_u^{1/3}J^{5/3}\big) \\
&\geq   -J^{-2}\beta_{v}^{2/3}h(\beta_{u},J)\left(\frac{1}{32\varepsilon_{0}} -C\right) \\
& =   -\beta_{u}^{2/3}J^{1/3}h(\beta_{u},J)\left(\frac{1}{32\varepsilon_{0}} -C\right).
\end{align*}  
Hence by choosing $\varepsilon_{0}$ small enough we get the result for any $J>0$. 
\end{proof}

\section{Large distance tail estimate}  \label{sec-large}
In this section we derive an upper bound for tail probability of
$R(T,J)$, which is given in the following proposition. Recall that 
$Q_{T}$ and $A^{(2)}_{T}$ were defined in \eqref{eq:def-Q} 
and \eqref{events}, respectively, and that $K_2$ was a constant appearing in
the definition of $A_2$.
\begin{prop}  \label{prop-large}
We can choose $K_2$ so large that 
\begin{equation*}
\lim_{T\to\infty}Q_{T}\big(A^{(2)}_{T}\big)=0.
\end{equation*}
\end{prop}
Before we start  with the proof of Proposition \ref{prop-large} we introduce 
the following large deviation result.  Let $\gamma >0$ and define 
$(X_t)_{t\geq 0}$ to be an Ornstein-Uhlenbeck process satisfying
\begin{equation} \label{eq:ON-process}
\begin{split}
dX_t&=-\gamma X_tdt+dW_t,  \\
X_0&=0
\end{split}
\end{equation} 
and let
\begin{equation*}
\mathcal S_T:=\int_{0}^{T}X_t^2dt.  
\end{equation*}
Lemma 3.1 of Bercu and Rouault \cite{BR01}, which relies on Bryc and 
Dembo \cite{Bryc:1997}, states the following.  
\begin{lemma} \label{lem:LD-square}
$T^{-1}\mathcal S_T$ satisfies a large deviation principle with good 
rate function
\begin{equation} 
\label{rate-f}
I(c)=
\begin{cases}
\frac{(2\gamma c-1)^2}{8c}  & \text{if $c>0$},  \\
+\infty & \text{otherwise}.
\end{cases}
\end{equation}
\end{lemma}

\begin{remark}
\label{rem:ldp}
In fact $I(c)=0$ if $c=(2\gamma)^{-1}$. As noted in (1.5) of Bercu 
and Rouault \cite{BR01},
\begin{equation*}
\lim_{T\to\infty}\frac{\mathcal{S}_T}{T}=\frac{1}{2\gamma}
\end{equation*}
almost surely.  So this lemma is only useful to us when $c>(2\gamma)^{-1}$.  
\end{remark}
\begin{proof}[Proof of Proposition \ref{prop-large}]
Once again we first prove the proposition for $J=1$, that is for a solution $v$ to \eqref{eq:edwards-wilkinson} on $x\in[0,1]$.   

Observe that since $G_t(u_0)(x)$ tends to a constant uniformly in $x$ as 
$t\to\infty$, we need only deal with $N(\cdot,\cdot)$.  So for the purposes 
of this proof, we will assume that $u_0$ is identically 0 and thus
$u(t,x)=N(t,x)$ for $t\ge0$, $x\in[0,J]$.  

Also recall that $\theta_u$ was defined in \eqref{theta} and that we use
$\theta_\varphi$ for a function $\varphi$ which may not depend on $t$. 
Recall that $(X^{(n)}_\cdot)_{n\geq 0}$ was defined in \eqref{ou-n}. 
We also define
\begin{equation*}
S^{(n)}_T=\int_{0}^{T}\big(X^{(n)}_t\big)^2dt.  
\end{equation*}
Recall that $N$ was defined in \eqref{N-def}, and that \eqref{n-decom} states
\begin{equation*}
N(t,x) = \sum_{n=0}^{\infty}X^{(n)}_t\varphi_n(x).
\end{equation*}
We define $\bar{N}(t)=\int_{0}^{1}N(t,x)dx$, as we defined $\bar{u}$. Since
$\varphi_0(x)=1$ for $x\in[0,1]$ and since the $(\varphi_n)_{n\ge0}$ are
orthogonal, we have $\bar{N}(t)=X^{(0)}_t\varphi_0(x)$,
for all $x\in[0,1]$. Therefore
\begin{equation*}
N(t,x)-\bar{N}(t,x) = \sum_{n=1}^{\infty}X^{(n)}_t\varphi_n(x).
\end{equation*}
Again using the fact that $(\varphi_n)_{n\ge1}$ are orthonormal, we get
\begin{equation}
\label{eq:orthogonal}
R_N(T,1)^2 = \sum_{n=1}^{\infty}\frac{1}{T}\int_{0}^{T}(X^{(n)}_t)^2dt
=\sum_{n=1}^{\infty}\frac{1}{T}S^{(n)}_T.
\end{equation}
We bound the final term in \eqref{eq:orthogonal} using the following lemma
.
We first observe that for a family of events $(A_T)_{T>0}$ that satisfies a large
deviations principle with rate function $I>0$ we have 
\begin{equation*}
\limsup_{T\to\infty}\frac{1}{T}\log P(A_T) \le -I.
\end{equation*}
Then,  for all $\kappa>1$ there exists $T_\kappa$
such that for $T>T_\kappa$ we have
\begin{equation*}
\frac{1}{T}\log P(A_T) \le -\frac{I}{\kappa} .
\end{equation*}
and so for $T>T_\kappa$
\begin{equation} \label{a-t-event}
P(A_T) \le \exp\left(-\frac{TI}{\kappa}\right).
\end{equation}
In our case, we have a family of events $(A_T^{(n)})_{n\ge1,T\ge0}$ where 
\begin{equation*}
A_T^{(n)}:=\left\{\frac{1}{T}S^{(n)}_T>Kn^{-2}\right\}.
\end{equation*}
So in principle we would have different times $T_\kappa^{(n)}$ in 
\eqref{a-t-event}. In the proof of the following lemma we use the fact 
that our processes $X^{(n)}_t$ are related in law and by scaling to a 
single Ornstein-Uhlenbeck process, so we can choose a single $T_\kappa$ 
which works for all $n$. For the next lemma, we choose $\kappa=2$.

\begin{lemma} \label{lem:uniform-large-dev}
There exists $\overline T>0$ such that if $T>\overline T$ then for all
$n\ge1$ we have
\begin{equation*}
P\left(\frac{1}{T}S^{(n)}_T>Kn^{-2}\right) 
  \le \exp\left(-\frac{TnI(K)}{2}\right),
\end{equation*}
where $I(\cdot)$ was defined in \eqref{rate-f}.
\end{lemma} 

Before proving this lemma, we use it to finish the proof of Proposition 
\ref{prop-large} in the case $J=1$. 

Let $\beta>0$ be a fixed constant. Recall that for every $K_{0}>0$, $K_2$ was defined in \eqref{k2-def}. 
Define 
\begin{equation} \label{c-0}  
c_0= \left(\sum_{n=1}^\infty \frac{1}{n^{2}}\right)^{-1} = \frac{6}{\pi^2}.
\end{equation} 
Now taking $c_{n}=2^{-3}c_0K_2^2n^{-2}$ and 
$\gamma=\lambda_n= n^2\pi^2$ in \eqref{rate-f}, we find that 
$c_{n}\gamma=2^{-3}c_0\pi^{2}K_2^2$. Following Remark \ref{rem:ldp}, in order to use Lemma \ref{lem:LD-square}, we need to show that 
\begin{equation} \label{req} 
2c_{n}\gamma=2^{-2}\pi^{2}c_0K_2^2 > 1, \quad \textrm{for all } n\geq 1. 
\end{equation}

We show that this requirement is satisfied in the following two 
cases. Case 1: for $\beta\geq e$, 
$K^{2}_{2}= \beta^{2/3}\log \beta K^{2}_0 \geq  K^{2}_0$.  
Note that we need $\beta> e$ here since $\log \beta $ could be arbitrarily small 
near $\beta =1$. Choosing $K^{2}_{2} > \frac{4}{\pi^{2}c_0}$ 
will be sufficient for \eqref{req}. Case 2: when $0<\beta<e $, 
$K^{2}_{2}= K^{2}_0$, and again 
$K^{2}_{2} > \frac{4}{\pi^{2}c_0}$ will be sufficient for 
\eqref{req}. 
From \eqref{rate-f}, \eqref{c-0} and \eqref{req} it follows that for $K_{0}$ (and therefore $K_{2}$) large enough, 
\begin{equation}  \label{i-c}
\begin{aligned}
I(c_{n})&=\frac{(2\gamma c_{n}-1)^2}{8c_{n}}  \\
&=\frac{(2^{-2}\pi^{2}c_0K_2^2 -1)^{2}}
 {  c_0K_2^2n^{-2} }  \\
 &\geq \frac{1}{2}n^2K_2^2, \quad \textrm{for all } n\geq 1.
\end{aligned}
\end{equation} 
From \eqref{eq:orthogonal} we get  
\begin{equation} \label{eq:p-r-sub-n}
\begin{split}
P\left(R_N(T,1) \geq K_{2}\right) &=
P\left(R_N(T,1)^2 \geq K_{2}^{2}\right) \\
&  \leq \sum_{n=1}^\infty P\left(\frac{1}{T}S^{(n)}_T>  c_0K_2^2n^{-2}\right),
\end{split}
\end{equation} 
where we have used \eqref{c-0} in the last inequality. 
Using this \eqref{eq:p-r-sub-n} together with Lemma 
\ref{lem:uniform-large-dev} and \eqref{i-c} give us that
\begin{equation*}
\limsup_{T\to\infty}P\left(R_N(T,1) \geq K_{2}\right) 
\le \limsup_{T\to\infty} \sum_{n=1}^{\infty}\exp\left(-TI(c_0K_{2}n^{-2})/2\right) = 0.
\end{equation*}

\begin{proof}[Proof of Lemma \ref{lem:uniform-large-dev}]
Let $(Y_t)_{t\ge0}$ satisfy \eqref{eq:ON-process} with $\gamma=1$, that is 
$Y_0=0$ and $dY=-Ydt+dB$. For $a,b>0$ let
\begin{equation*}
\begin{split}
Z^{(a,b)}_t &= aY_{bt}  \\
B^{(a,b)}_t &= aB_{bt}.
\end{split}
\end{equation*}
Then we find
\begin{equation*}
\begin{split}
dZ^{(a,b)}_t &= d(aY_{bt})  \\
&= -aY_{bt}d(bt) + adB_{bt}  \\
&= -bZ^{(a,b)}_t + dB^{(a,b)}_t
\end{split}
\end{equation*}
and also $Z^{(a,b)}_0=0$.

Choosing $b=\lambda_n=\pi^2n^2$ and $a=(\pi n)^{-1}$ and writing 
$Y^{(n)}=Z^{(a,b)}$, we see that $B^{(n)}:=B^{(a,b)}$ is a standard Brownian
motion. So for $t\ge0$
\begin{equation*}
dY^{(n)}_t = -\lambda_n Y^{(n)}dt + dB^{(n)}_t
\end{equation*}
and $(X^{(n)}_t)_{t\ge0}$ and $(Y^{(n)}_t)_{t\ge0}$ are equal in distribution. 
Then Lemma \ref{lem:LD-square} implies that
\begin{equation*}
\begin{split}
P\left(\frac{1}{T}S^{(n)}_T>Kn^{-2}\right)
&= P\left(\frac{1}{T}\int_{0}^{T}(X^{(n)}_t)^2dt>Kn^{-2}\right)  \\
&=  P\left(\frac{1}{T}\int_{0}^{T}(Y^{(n)}_t)^2dt>Kn^{-2}\right)  \\
&=  P\left(\frac{1}{T}\int_{0}^{T}
        \left((\pi n)^{-1}Y_{\pi^2n^2t}\right)^2dt>Kn^{-2}\right)  \\
&=  P\left(\frac{1}{\pi^2n^2T}\int_{0}^{\pi^2n^2T}
        \left(Y_s\right)^2ds>\pi^2K\right),
\end{split}
\end{equation*}
where we have used the change of variable $s=\pi^2n^2t$.  

Then Lemma \ref{lem:LD-square} finishes the proof of Lemma 
\ref{lem:uniform-large-dev}. As we remarked, this also finishes the proof of 
Lemma \ref{prop-large} when $J=1$.  
\end{proof}

Now use scaling to finish the proof of Lemma \ref{prop-large} for all $J>0$.  
Let $u$ be the solution to \eqref{eq:edwards-wilkinson} 
on $x\in[0,J]$.  Recall that $v$ is the solution to 
\eqref{eq:edwards-wilkinson} on $x\in[0,1]$. Now we use Lemma 
\ref{lem-scale}(ii) and (iii) together with \eqref{h-scale} 
to get along the same lines as \eqref{scale_trans} that 
\begin{align*} 
Q_{TJ^{-2},1,\beta_{v}}&\big(  R_v(TJ^{-2},1)\leq K_0h(\beta_{v},1)^{1/2}\beta_{v}^{1/3}\big)  \\
& =Q_{T,J,\beta_{u}}\big( R_u(T,J)\leq K_0h(\beta_{u},J)^{1/2}\beta_u^{1/3}J^{5/3}\big), 
\end{align*}  
where $\beta_u = J^{-7/2}\beta_v^{} $.  
From \eqref{k2-def}, \eqref{events} and Proposition \ref{prop-large}(i) with $J=1$ we get that there exists $K_{0}>0$ such that for any $\beta_{v} \geq e$ we have,
\begin{equation*}
\lim_{T\to\infty}Q_{T,1,\beta_{v}} \big( 
 R_v(T,1)\leq K_0 h(\beta_{v},1)^{1/2}\beta_{v}^{1/3} \big)=1.
\end{equation*}
It follows that 
\begin{equation*}
\lim_{T\to\infty} Q_{T,J,\beta_{u}}\big( R_u(T,J)\leq K_0\beta_{u}^{1/3}h(\beta_{u},J)^{1/2}J^{5/3}\big)=1, 
\end{equation*}
for $\beta_u = J^{-7/2}\beta_v^{} \geq eJ^{-7/2}$, so we get (i) for all $J>0$. 
 \medskip \\
From Lemma \ref{lem:uniform-large-dev} with $J=1$, we get that there exists 
$K_{0}>0$ such that for any $0<\beta_{v} \leq e$ we have,
\begin{equation*}
\lim_{T\to\infty}Q_{T,1,\beta_{v}} \big( 
 R_v(T,1)\leq K_0 \big)=1.
\end{equation*}
Therefore, using the same scaling argument, for any $J\geq 1$ and 
$\beta_u  \leq eJ^{-7/2}$ we have 
\begin{equation*}
\lim_{T\to\infty} Q_{T,J,\beta_{u}}\big( R_u(T,J)\leq \tilde K_{0}J^{5/3}\big)=1, 
\end{equation*}
where  $\tilde K_{0}= K_0 J^{-7/6} \leq K_{0}$. Then we get (i) for all $J \geq 1$.

This finishes the proof of Proposition \ref{prop-large}.  
\end{proof} 
 
\section{Scaling of the Polymer}\label{section-scale}
 
This section is dedicated to the proof of Lemma \ref{lem-scale}. 
Throughout this section we assume that $u(t,x)$ satisfies equation \eqref{eq:edwards-wilkinson} on $x\in[0,J]$ and that $v$ is defined as in \eqref{def-v}.

\begin{proof}[Proof of Lemma \ref{lem-scale}]
(i) This is a well known property of the stochastic heat equation, 
but we give a short derivation for completeness.  First, recall that
$(G^J_t(x))_{t\ge0,x\in[0,J]}$ is the Neumann heat kernel on $[0,J]$, and 
that the following scaling equalities hold.  
\begin{gather*}
G^1_{J^{-2}t}(J^{-1}x,J^{-1}y)=J\cdot G^J_t(x,y),  \\
W^{J^{2}, J}(dy\, ds) := J^{-3/2}W(Jdy,J^2ds)
    \stackrel{\mathcal{D}}{=}W(dyds),
\end{gather*}
where $\stackrel{\mathcal{D}}{=}$ means equality in distribution.  
The first equality above follows from \eqref{eq:neumann-expansion} 
and the scaling properties of the heat kernel on $\mathbf{R}$.  
Then by the definition \eqref{def-v} of $v(t,x)$,
\begin{align*}  
v(t,z)&=J^{-1/2}u(J^2t,Jz)  \\
&=J^{-1/2}\int_{0}^{J}G_J(J^2t,Jz-y)u_0(y)dy  \\
&\quad +J^{-1/2}\int_{0}^{J^2t}\int_{0}^{J}G_J(J^2t-s,Jz-y)W(dyds)\\
&= J^{-3/2}\int_{0}^{J}G_1(t,z-J^{-1}y)u_0(y)dy  \\
&\quad +J^{-3/2}\int_{0}^{J^2t}\int_{0}^{J}G_1(t-J^{-2}s,z-J^{-1}y)W(dyds)\\
&= \int_{0}^{1}G_1(t,z-w)v_0(w)dw  \\
&\quad +\int_{0}^{t}\int_{0}^{1}G_1(t-r,z-w)W^{J^2, J}(dwdr),
\end{align*}
This proves part (i).\\

(ii) We first derive the relation between $\ell^u_t$ and $\ell^v_t$. 
From \eqref{l-time} and \eqref{def-v} we get 
\begin{align*}
\ell^v_t(y)&=\partial_y\int_{0}^{1}\mathbf{1}_{(-\infty,y]}(v(t,x))dx \\
&=
 \partial_y\int_{0}^{1}\mathbf{1}_{(-\infty,y]}(J^{-1/2}u(J^2t,Jx))dx \\
&=\partial_y\int_{0}^{1}\mathbf{1}_{(-\infty,J^{1/2}y]}(u(J^2t,Jx))dx.
\end{align*}
Changing variables to $z=Jx$, $w=J^{1/2}y$ we get
\begin{align*}
\ell^v_t(y)
&=J^{1/2}\partial_w\int_{0}^{J}\mathbf{1}_{(-\infty,w]}(u(J^2t,z))J^{-1}dz \\
&=J^{-1/2}\ell^u_{J^2t}(w).
\end{align*}
It follows that 
\begin{equation*}
\ell^u_t(w)=J^{1/2}\ell^v_{J^{-2}t}(y)=J^{1/2}\ell^v_{J^{-2}t}(J^{-1/2}w). 
\end{equation*}
Thus,  
\begin{align*}
\int_{0}^{T}\int_{-\infty}^{\infty}\ell^u_t(w)^2dwdt
&=J\int_{0}^{T}\int_{-\infty}^{\infty}
 \ell^v_{J^{-2}t}(J^{-1/2}w)^2dwdt.
\end{align*}
Making the change of variable $s=J^{-2}t$ and going back to the original variable $w=J^{1/2}y$, we get
\begin{equation}  \label{scale-sq-lt}
\begin{aligned}
\int_{0}^{T}\int_{-\infty}^{\infty}\ell^u_t(w)^2dwdt
&=J\int_{0}^{J^{-2}T}\int_{-\infty}^{\infty}
 \ell^v_{s}(y)^2J^{1/2}dyJ^2ds   \\
&=J^{7/2}\int_{0}^{J^{-2}T}\int_{-\infty}^{\infty}
 \ell^v_{s}(w)^2dyds. 
\end{aligned}
\end{equation} 
Let $\beta>0$. From \eqref{z-func} and \eqref{scale-sq-lt} we get 
\begin{equation*}
\begin{aligned}
\mathcal{E}^{u}_{T,J,\beta_{u}}&= \exp\left(-\beta \int_{0}^{T}\int_{-\infty}^{\infty}\ell^u_t(y)^2dydt\right) \\
&=
\exp\left( -\beta J^{7/2}\int_{0}^{J^{-2}T}\int_{-\infty}^{\infty}
 \ell^v_{s}(w)^2dyds \right) \\
 &=\mathcal{E}^{v}_{J^{-2}T,1,\beta J^{7/2}}. 
 \end{aligned}
\end{equation*}
From the above equation together with \eqref{eq:def-Q} we get 
$$
 \quad Q^{u}_{T,J,\beta }(\cdot) =Q^{v}_{TJ^{-2},1,\beta J^{7/2} }(\cdot),
  $$
and we have proved (ii). \\

(iii) Recall the definitions of $R_v(T,1)$ and $R_u(J^2T,J)$ from \eqref{eq:H-def}. From \eqref{u-bar} and \eqref{def-v} we get 
\begin{equation} \label{v-bar-sc}
\begin{aligned} \bar v(t) &= \int_0^1v(t,x)dx \\
&=J^{-1/2}\int_0^1u(J^2 t,Jx)dx \\
&=J^{-3/2}\int_0^Ju(J^2 t,x)dx \\
&=J^{-1/2} \bar u(J^2 t,x).
\end{aligned} 
\end{equation}

Using \eqref{def-v}, \eqref{v-bar-sc} and making the change of variables $s=J^2t$, $y=Jx$, we get 
\begin{align*}
R_v(T,1)^2 &=\frac{1}{T}\int_{0}^{T}\int_0^1\Big( v(t,x)
 - \bar v(t)\Big)^2dxdt  \\
& = \frac{1}{TJ}\int_{0}^{T}\int_0^1\Big( u(J^2t,Jx)
 - \bar u(J^2t)\Big)^2dxdt  \\
& = \frac{1}{TJ^4}\int_{0}^{TJ^{2}}\int_0^J\Big( u(s,y)
 - \bar u(s)\Big)^2dydt  \\
&= J^{-1}R_u(J^2T,J)^2.
\end{align*}
Taking the square root and defining $\tilde T = TJ^{-2}$, we have
\begin{equation*}
R_u(\tilde T,J) =J^{1/2}R_v(J^{-2}\tilde T,1), 
\end{equation*}
and we get (iii). 
\end{proof} 

\section{Properties of the pinned string} \label{pf-prop-string}
This section we prove Lemmas \ref{lem-conCD}, \ref{l-t-ident} and \ref{lemma-shift}. 

\begin{proof} [Proof of Lemmas \ref{lem-conCD}]

(i) Recall that $(C)$ was defined in \eqref{eq:u-drift}. Using Ito's
isometry, \eqref{phi-n}, \eqref{g-exp}, Fubini's theorem and the fact that 
$\{\varphi_{n}\}_{n\geq 0}$ is an orthonormal basis, we get
\begin{equation} \label{rt1}
\begin{aligned}
&\hat{E}[(C)^2]  
= \int_{-\infty}^{t_0}\int_{0}^{1}\big[G_{t-r}(x,y)-G_{t_0-r}(x_0,y)\big]^2dydr   \\
&=\int_{-\infty}^{t_0}\sum_{n=0}^{\infty}\sum_{j=0}^{\infty}
 \left[e^{-(t-r)\lambda_n}\varphi_n(x)-e^{-(t_0-r)\lambda_n}\varphi_n(x_0)\right]  \\ 
 & \qquad \times  \left[e^{-(t-r)\lambda_j}\varphi_j(x)-e^{-(t_0-r)\lambda_j}\varphi_j(x_0)\right]\int_{0}^{1}\varphi_n(y)\varphi_j(y) dydr \\ 
&=\int_{-\infty}^{t_0}\sum_{n=0}^{\infty}
 \left[e^{-(t-r)\lambda_n}\varphi_n(x)-e^{-(t_0-r)\lambda_n}\varphi_n(x_0)\right]^2dr \\
&\leq2\int_{0}^{\infty}\sum_{n=1}^{\infty}
 \left[e^{-(t-t_{0}+r)\lambda_n}+e^{-r\lambda_n}\right]^2dr,
\end{aligned}
\end{equation}
where we dropped the $n=0$ term because $\lambda_0=0$ and $\varphi_0\equiv 1$, so the difference in the brackets is zero.  
Also, we used the fact that $\|\varphi_n\|_\infty^2\leq2$ for all $n\geq 1$, which follows from \eqref{phi-n}.  
It follows that 
\begin{equation*}
\hat{E}[(C)^2] \leq2\sum_{n=0}^{\infty}\int_{0}^{\infty}
 \left[2e^{-r\lambda_n}\right]^2dr 
\leq8\sum_{n=0}^{\infty}(2\lambda_n)^{-1} 
<\infty, 
\end{equation*}
where we used \eqref{phi-n} in the last inequality.  \\

(ii)
From \eqref{hk-int}, we get 
\begin{align*}
|(D)|&= \left|\int_{-\infty}^{t_0}\int_{0}^{1}
 \big[G_{t-r}(x,y)-G_{t_0-r}(x_0,y)\big]\varphi_1(y)dydr\right|   \\
&=\left|\int_{0}^{\infty}\int_{0}^{1}
 \big[G_{t-t_0+r}(x,y)-G_{r}(x_0,y)\big]\varphi_1(y)dydr\right|   \\
&\leq \int_{0}^{\infty}\left[e^{-\lambda_1(t-t_{0}+r)}+e^{-\lambda_1r}\right]
 |\varphi_1(x)| dr  \\
&<\infty.  
\end{align*}
\end{proof}

\begin{proof} [Proof of Lemma \ref{l-t-ident}]
Let $\hat{f}_{t,x_1,x_2}$ be the joint 
probability density of\\
$(u(t,x_1),u(t,x_2))$ under $\hat P_T^{(a)}$.   

Let $S_1,S_2\subset\mathbf{R}$.  Then from \eqref{l-time} we have
\begin{align*}
\int_{-\infty}^{\infty}\int_{-\infty}^{\infty}
 \mathbf{1}_{S_1}(y_1)&\mathbf{1}_{S_2}(y_2)
 \ell_t(y_1)\ell_t(y_2)dy_2dy_1  \\
&= \int_{0}^{1}\int_{0}^{1}\mathbf{1}_{S_1}(u(t,x_1))
 \mathbf{1}_{S_2}(u(t,x_2))dx_2dx_1. 
\end{align*}
Therefore
\begin{align*}
\hat{E} \bigg[\int_{-\infty}^{\infty}&\int_{-\infty}^{\infty}
 \mathbf{1}_{S_1}(y_1)\mathbf{1}_{S_2}(y_2)
 \ell_t(y_1)\ell_t(y_2)dy_2dy_1 \bigg] 
  \\
&= \int_{0}^{1}\int_{0}^{1}\int_{-\infty}^{\infty}\int_{-\infty}^{\infty}
 \mathbf{1}_{S_1}(y_1)\mathbf{1}_{S_2}(y_2)\hat{f}_{t,x_1,x_2}(y_1,y_2)
 dy_2dy_1dx_2dx_1.
\end{align*}
By using Dynkin's monotone class theorem we can move from 
indicator functions to simple functions to nonnegative measurable 
functions to bounded measurable functions 
$(h(y_1,y_2))_{y_1,y_2\in\mathbf{R}}$, we get
\begin{align*}
\hat{E}\bigg[\int_{-\infty}^{\infty}\int_{-\infty}^{\infty}
 &h(y_1,y_2)\ell_t(y_1)\ell_t(y_2)dy_2dy_1 \bigg] \\
&= \int_{0}^{1}\int_{0}^{1}\int_{-\infty}^{\infty}\int_{-\infty}^{\infty}
h(y_1,y_2)\hat{f}_{t,x_1,x_2}(y_1,y_2)dy_2dy_1dx_2dx_1.
\end{align*}
Finally, replacing $h(y_1,y_2)$ by an approximate identity 
$a(y_1-y_2)$ and taking a limit, we get
\begin{equation*}
\hat{E}\bigg[\int_{-\infty}^{\infty}\ell_t(y)^2dy  \bigg]
= \int_{0}^{1}\int_{0}^{1}\int_{-\infty}^{\infty}
 \hat{f}_{t,x_1,x_2}(y,y)dydx_2dx_1 .
\end{equation*}

Recall that $\hat{g}_{t,x_1,x_2}$ is the probability density function of 
$u(t,x_2)-u(t,x_1)$ under $\hat P_T^{(a)}$. Then by elementary probability,
\begin{equation*}
\hat{g}_{t,x_1,x_2}(z)=\int_{-\infty}^{\infty}\hat{f}_{t,x_1,x_2}(y,z+y)dy
\end{equation*}
and we get that 
\begin{equation*}
\hat{E}\bigg[\int_{-\infty}^{\infty}\ell_t(y)^2dy  \bigg]
= \int_{0}^{1}\int_{0}^{1}\hat{g}_{t,x_1,x_2}(0)dx_2dx_1.
\end{equation*}
\end{proof} 

\begin{proof} [Proof of Lemma \ref{lemma-shift}] 
Let $0\leq t_{0} \leq T$ and $x_{0}\in[0,1]$. From \eqref{eq:u-drift} 
we have for any $t_{0} \leq t \leq T$ and $x_{1}, x_{2} \in [0,1]$, 
\begin{equation*}  
\begin{aligned} 
u_{t_0,x_0}(&t,x_{1})-u_{t_0,x_0}(t,x_{2})\\
&=\int_{t_0}^{t}\int_{0}^{1}\big(G_{t-r}(x_{1},y)-G_{t-r}(x_{2},y)\big)\hat{W}(dydr)  \\
&\quad +a\int_{t_0}^{t}\int_{0}^{1}\big(G_{t-r}(x_{1},y)-G_{t-r}(x_{2},y)\big)\varphi_1(y)dydr  \\
&\quad+\int_{-\infty}^{t_0}\int_{0}^{1}\big(G_{t-r}(x_{1},y)-G_{t-r}(x_{2},y)\big)\hat{W}(dydr)   \\
&\quad+a\int_{-\infty}^{t_{0}}\int_{0}^{1}\big(G_{t-r}(x_{1},y)-G_{t-r}(x_2,y)\big)\varphi_1(y)dydr  \\
&=\int_{-\infty}^{t}\int_{0}^{1}\big(G_{t-r}(x_{1},y)-G_{t-r}(x_{2},y)\big)\hat{W}(dydr)  \\
&\quad+a\int_{-\infty}^{t}\int_{0}^{1}\big(G_{t-r}(x_{1},y)-G_{t-r}(x_2,y)\big)\varphi_1(y)dydr  \\
&=: \mathcal{J}_{1}(t,x_{1},x_{1}) + a\mathcal{J}_{2}(t,x_{1},x_{1}).
\end{aligned} 
\end{equation*} 
 
Note that $\mathcal{J}_{1}(t,\cdot,\cdot)$ is a centered Gaussian random field with the following covariance functional 
\begin{equation*}  
C_{t}(x_{1},x_{2}) :=E\left[\int_{-\infty}^{t}\int_{0}^{1}(G_{t-r}(x_{1},y)-G_{t-r}(x_{2},y))^{2}dydr \right]. 
\end{equation*}
By a change of variable $s= t-r$ we get 
\begin{equation}  
\begin{aligned} 
C_{t}(x_{1},x_{2})  &= E\left[\int_{0}^{\infty}\int_{0}^{1}(G_{s}(x_{1},y)-G_{s}(x_{2},y))^{2}dyds \right]  \\ 
 &=C_{0}(x_{1},x_{2}). 
 \end{aligned} 
\end{equation}
So we have 
\begin{equation*}
\mathcal{J}_{1}(t,\cdot,\cdot) \stackrel{\mathcal{D}}{=} \mathcal{J}_{1}(0,\cdot,\cdot). 
\end{equation*}
We also get by the change of variable $s= t-r$ that 
\begin{align*} 
\mathcal{J}_{2}(t,x_{1},x_{2}) &=a\int_{-\infty}^{t}\int_{0}^{1}\big[G_{t-r}(x_{1},y)-G_{t-r}(x_2,y)\big]\varphi_1(y)dydr  \\ 
&=a\int_{0}^{\infty}\int_{0}^{1}\big[G_{s}(x_{1},y)-G_{s}(x_2,y)\big]\varphi_1(y)dydr  \\ 
&=\mathcal{J}_{2}(0,x_{1},x_{2}), 
\end{align*} 
and Lemma \ref{lemma-shift} follows. 
\end{proof} 

\section{Proofs of Lemmas \ref{lemma:2nd-mom}--\ref{int-conv}}\label{pf-int-conv}
 
\begin{proof} [Proof of Lemma \ref{lemma:2nd-mom}] 
The proof of the upper bound is standard (see e.g. Proposition 3.7 in Chapter III.4 of \cite{wal86}).  

Next we prove the lower bound for
\begin{equation*}
E\left[\left(u(t,x_1)-u(t,x_2)\right)^2\right] 
 = \int_{0}^{\infty}\int_{0}^{1}\left[G_t(x_1,y)-G_t(x_2,y)\right]^2dydt.
\end{equation*}
From \eqref{phi-0}, \eqref{phi-n} and \eqref{g-exp}  we have 
\begin{equation} \label{g-dec} 
G_t(x,y)=1+2\sum_{n=1}^{\infty}e^{-\pi^2n^2t}
 \cos(\pi nx)\cos(\pi ny).
\end{equation}
By repeating similar steps as in \eqref{rt1} and using the orthogonality 
of $(\varphi_n)_{n\ge1}$ we get
\begin{equation}  \label{l-1}
\begin{split} 
\mathcal{J}:=&\int_{0}^{\infty}\int_{0}^{1}\left[G_t(x_1,y)-G_t(x_2,y)\right]^2dydt \\
 &= 2\sum_{n=1}^{\infty}\int_{0}^{\infty}
 e^{-2\pi^2n^2 t}\left[\cos(\pi nx_1)-\cos(\pi nx_2)\right]^2dt  \\
&= \frac{1}{\pi^2}\sum_{n=1}^{\infty}\frac{1}{n^{2}}
 \left[\cos(\pi nx_1)-\cos(\pi nx_2)\right]^2.
\end{split} 
\end{equation} 
Let $x=(x_1+x_2)/2$ and $h=x_2-x_1$. It suffices to prove our estimate for 
$h<\delta_0$ where $\delta_0>0$ is some small number to be chosen later.  
By symmetry we may assume $x\in[0,1/2]$, otherwise we reflect the 
configuration about $1/2$.  Also note that $0\le x_1=x-h/2$ so $x\ge h/2$.
Let $\delta_1>0$ be a small number to be chosen later, and let 
$M=[2h^{-1}(1-\delta_1)]$ where $[\cdot]$ denotes the greatest integer 
function.  

Recall the following trigonometric identities.  
\begin{equation*}
\begin{split}
\cos(a)-\cos(b) 
&= -2\sin\left(\frac{a-b}{2}\right)\sin\left(\frac{a+b}{2}\right)  \\
\sum_{n=1}^{M}\sin^2(nx)
&= \frac{1}{4}\left[2M+1-\frac{\sin((2M+1)x)}{\sin(x)}\right]
\end{split}
\end{equation*}

Continuing to estimate the above integral $\mathcal{J}$, we get (since $h<1/4$) 
and using the above trigonometric identities,
\begin{equation*}
\begin{split}
\mathcal{J} &= \frac{4}{\pi^2}\sum_{n=1}^{\infty}\frac{1}{n^2}
  \sin^2\left(n\pi2^{-1}h\right)\sin^2\left(n\pi x\right)  \\
&\ge \frac{4}{\pi^2}\sum_{n=1}^{M}\frac{1}{n^2}
  \sin^2\left(n\pi2^{-1}h\right)\sin^2\left(n\pi x\right)  \\
&\ge Ch^2\sum_{n=1}^{M}\sin^2\left(n\pi x\right)  \\
&= Ch^2\left[2M+1-\frac{\sin((2M+1)\pi x)}{\sin(\pi x)}\right] 
\end{split}
\end{equation*}
For the second to the third line, we require
\begin{equation} \label{eq:first-bound}
\pi(1-\delta_1)-1\le M\pi2^{-1}h\le \pi(1-\delta_1)
\end{equation}
which holds by the definition of $M$.  This implies 
that for $1\le n\le M$ we have $\sin(n\pi2^{-1}h)\ge cnh$ for 
some $c>0$ not depending on $n$. 

Now we wish to show that the term in square brackets above is of order $M$.  
We need to show that for some small number $\delta_2>0$ to be chosen later,
\begin{equation} \label{eq:need-to-show}
\frac{\sin((2M+1)\pi x)}{\sin(\pi x)}\le 2M(1-\delta_2).
\end{equation}
First, $\sin((2M+1)\pi x)\le1$.
We also want a lower bound on $\sin(\pi x)$.  Recall that $h/2\le x\le1/2$ 
and $h<\delta_0$.  Since $\sin(x)$ is increasing on $x\in[h/2,1/2]$,
if $\delta_3>0$ then we can choose $\delta_0$ small enough so that
\begin{equation*}
\sin(\pi x)\ge \sin(\pi2^{-1}h) \ge \pi2^{-1}h(1-\delta_3).
\end{equation*}
Thus if we assume equality holds in \eqref{eq:first-bound} then we have
\begin{equation*}
\begin{split}
\frac{\sin((2M+1)\pi x)}{\sin(\pi x)}
&\le \frac{1}{\pi2^{-1}h(1-\delta_3)}  \\
&= M\cdot\frac{1}{(M\pi2^{-1}h)(1-\delta_3)}  \\
& \leq M\cdot\frac{1}{[\pi(1-\delta_1)-1](1-\delta_3)},
\end{split}
\end{equation*}
which verifies \eqref{eq:need-to-show}, provided $\delta_1,\delta_2,\delta_3$ 
are small enough.  

The end result is that for $\delta_2$ small enough, there exists 
$\widetilde C>0$ such that,
\begin{equation*}
\mathcal{J}\ge Ch^2M\delta_2 \geq 2\widetilde Ch\delta_2 = 2\widetilde C|x_1-x_2|\delta_2,  \quad \textrm{for all } |x_1-x_2| \leq \delta_{0},
\end{equation*}
where we have used \eqref{eq:first-bound} in the second inequality. 
\end{proof}

\begin{proof}[Proof of Lemma \ref{lem-drift}]
Let $0\leq x_{1} \leq x_{2}\leq 1$. Then we have 
\begin{equation*}
\begin{aligned}  
\mathcal D(x_{1},x_{2})&:=(F)(x_{1}) -(F)(x_{2}) \\
&=\frac{\sqrt{2}}{\pi^{2}}\big(\cos(\pi x_{1}) -\cos(\pi x_{2}) \big) \\
&= \frac{2\sqrt{2}}{\pi^{2}} \sin\left(\frac{\pi}{2} (x_{1}+x_{2})\right) \sin\left(\frac{\pi}{2}(x_{2}-x_{1})\right). 
\end{aligned} 
\end{equation*}
We will use the following lower bounds on the sine function: 
\begin{align*} 
\sin(\pi x) &\geq \frac{ x}{2}, \quad \textrm{for all } 0\leq x\leq \frac{1}{2}, \\ 
\sin(\pi x) & \geq  1-x, \quad \textrm{for all } \frac{1}{2}\leq x\leq 1. 
\end{align*}
We get that 
$$
 \sin\left(\frac{\pi}{2}(x_{2}-x_{1})\right) \geq \frac{x_{2} -x_{1}}{4}, \quad \textrm{for all } 0\leq x_{1} \leq x_{2} \leq 1.
$$
For $0\leq x_{1} \leq x_{2} \leq 1$ we also have 
$$
 \sin\left(\frac{\pi}{2}(x_{2}+x_{1})\right) \geq 
 \begin{cases}
         \frac{x_{2} +x_{1}}{4}, & \text{for }   x_{1}+x_{2} \leq 1 , \\ 
         1-\frac{x_{2}+x_{1}}{2}, & \text{for }   1< x_{1}+x_{2} \leq 2. 
        \end{cases}
$$
We get that for all $0\leq x_{1} \leq x_{2} \leq 1$ 
$$
 \mathcal D(x_{1},x_{2}) \geq 
 \begin{cases}
           \frac{\sqrt{2}}{8\pi^{2}} (x_{2} -x_{1})(x_{2} +x_{1}), & \text{for }   x_{1}+x_{2} \leq 1 , \\ 
         \frac{\sqrt{2}}{2\pi^{2}} (x_{2} -x_{1})\left(1-\frac{x_{2}+x_{1}}{2}\right), & \text{for }   1< x_{1}+x_{2} \leq 2. 
        \end{cases}
$$
\end{proof} 

\begin{proof}[Proof of Lemma \ref{int-conv}]

First we bound $\mathcal I_{2}(a)$. Recall that
\begin{align*}
\mathcal I_{2}(a) =&  \int_{0}^{1}\int_{x_{1}}^{1} \mathbf{1}_{\{x_{1}+x_{2}
> 1\}} (x_{2}-x_{1})^{-1/2} \\
&\qquad \qquad \times\exp\left(-C_{4}a^2\left(1- \frac{x_{2}+x_{1}}{2}\right)^2(x_{2}-x_{1})\right)dx_{1}dx_{2}.
\end{align*}
Using the transformation 
\begin{align*} 
x'_{2}-x'_{1} &= x_{2}-x_{1}, \\
x'_{2}+x'_{1} &= 1-\frac{x_{2}+x_{1}}{2},
\end{align*} 
we get that 
\begin{align*}
&\mathcal I_{2}(a) \\
 &\leq  \frac{1}{2}\int_{-1}^{1}\int_{x_{1}'}^{1}  (x'_{2}-x'_{1})^{-1/2} \exp\left(-C_{4}a^2\left( x'_{1}+ x'_{2}\right)^2(x'_{2}-x'_{1})\right)dx'_{1}dx'_{2}, 
\end{align*}
where the inequality is due to an enlargement of the integration region. 

Letting $x'_{1}=x$ and $x'_{2}-x'_{1} = h$ we get that 
\begin{equation} \label{i1-2} 
\mathcal I_{2}(a) \leq  \frac{1}{2}\int_{-1}^{1}\int_{0}^{1} h^{-1/2} \exp\left(-C_{4}a^2\left( 2x+h\right)^2h\right)dh dx.
\end{equation}

We distinct between the following cases.  \\

\textbf{Case I: $0\leq a\leq1$.}

Let $y=2a^{2/3}x$, $g=a^{2/3}h$.  Then we get from \eqref{i1-2} that 
\begin{equation*}   
\begin{aligned}
\mathcal I_{2}(a) &\leq \frac{1}{4a}\int_{-a^{2/3}}^{a^{2/3}}\int_{0}^{2a^{2/3}} g^{-1/2}\exp\left(-\frac{C_{4}}{2}(y+g)^2g\right)dgdy \\
&\leq  \frac{1}{4a}  \int_{-a^{2/3}}^{a^{2/3}} \int_{0}^{2a^{2/3}}g^{-1/2}dgdy  \\ 
&\leq C \frac{1}{a}, 
\end{aligned}
\end{equation*} 
where $C$ does not depend on $a$. We therefore derived the bound (i) for $\tilde{ \mathcal I_{2}}(a) $. \\

\textbf{Case II: $a>1$.}

We split the right-hand side of \eqref{i1-2} into two regions:
\begin{equation} \label{i1-10} 
\begin{aligned}  
\frac{1}{2}\int_{-1}^{1}&\int_{0}^{1} h^{-1/2} \exp\left(-C_{4}a^2\left( 2x+h\right)^2h\right)dh dx \\
&= \frac{1}{2}\int_{-1}^{1}\int_{0}^{\frac{1}{a^{2}}} h^{-1/2} \exp\left(-C_{4}a^2\left( 2x+h\right)^2h\right)dh dx \\
&\quad+\frac{1}{2}\int_{-1}^{1}\int_{\frac{1}{a^{2}}}^{1} h^{-1/2} \exp\left(-C_{4}a^2\left( 2x+h\right)^2h\right)dh dx \\ 
&=: \frac{1}{2} \big(\mathcal I_{2,1}(a)+\mathcal I_{2,2}(a)\big).
\end{aligned} 
\end{equation}
We first deal with $\mathcal I_{2,1}(a)$. By making the same change of variable as in case I, that is $y=2a^{2/3}x$ and $g=a^{2/3}h$, we get that 
\begin{equation}  \label{int-11}
\begin{aligned} 
\mathcal I_{2,1}(a) &=\frac{1}{2a} \int_{-2a^{2/3}}^{2a^{2/3}}\int_{0}^{a^{-4/3}}g^{-1/2}\exp\left(-C_{4}(y+g)^2g\right)dgdy \\
&\leq \frac{1}{2a} \int_{-2a^{2/3}}^{2a^{2/3}}\int_{0}^{a^{-4/3}}g^{-1/2} dgdy \\
&= \frac{1}{2a} \int_{-2a^{2/3}}^{2a^{2/3}}2a^{-2/3} dy \\
&= \frac{4}{a}.
\end{aligned} 
\end{equation}

Now we deal with $\mathcal I_{2,2}(a)$. 
Define 
$$
\sigma^{2} = \frac{1}{8C_{4}ha^{2}}, 
$$
and using the Gaussian density we get 
\begin{align*} 
 \int_{-\infty }^{\infty} \exp\left(-C_{4}a^2\left( 2x+h\right)^2h\right)dx &=\int_{-\infty }^{\infty}   \exp\left(-\frac{\left( x+h/2\right)^2}{2\sigma^{2}}\right)dx \\
 & = \sqrt{2\pi} \sigma. 
\end{align*} 
It follows that 
\begin{equation}  \label{int-51} 
\begin{aligned}
 \mathcal I_{2,2}(a) &= \int_{\frac{1}{a^{2}}}^{1}h^{-1/2} \int_{-\infty }^{\infty}    \exp\left(-C_{4}a^2\left( 2x+h\right)^2h\right)dx dh  \\
  &=  \int_{\frac{1}{a^{2}}}^{1}h^{-1/2} \sqrt{2\pi} \sigma dh \\ 
 &=\sqrt{\frac{2\pi}{ 8C_{4}}} \frac{1}{a}\int_{\frac{1}{a^{2}}}^{1}h^{-1}  dh \\ 
 &= C\frac{\log a}{a}, 
\end{aligned}
\end{equation} 
where $C$ is independent of $a$. From \eqref{i1-2}--\eqref{int-51} we get (ii) for $ \mathcal I_{2}(a)$. 
\\

Next, we bound $\mathcal I_{1}(a)$. By setting $h=x_2-x_1$ and $x=x_1$ we get 
\begin{align*} 
\mathcal I_{1}(a) & \leq  \int_{0}^{1}\int_{0}^{1}\mathbf{1}_{\{2x+h \leq
1\}} h^{-1/2}
 \exp\left(-C_{3}a^2(x_{1}+x_{2})^2h\right)dhdx 
\\
& \leq\int_{0}^{1}\int_{0}^{1} h^{-1/2}
 \exp\left(-C_{3}a^2(x_{1}+x_{2})^2h\right)dhdx .
\end{align*} 
We can therefore bound $\mathcal I_{1}(a)$ similarly as we bounded the right hand side of \eqref{i1-2}, so we are done. 
\end{proof}

\appendix
\label{appendix}

\section{Some heuristic ideas}
For intuitive purposes, we define $R$ so that the processes in this section are restricted to a ball of radius $R$. 
\subsection{Flory's argument for self-avoiding walk in two dimensions}
\label{subsec:Flory}
Since this is an imprecise argument, we will freely assume that 
the walk $(S_n)_{n\in\mathbf{N}_0}$ is a Brownian motion $B(t)$. 
Large deviation theory suggests that probabilities such as $Q_T(A)$ 
are dominated by a ``most likely path'' $(X(t))_{t\in[0,T]}$, 
with approximate probability
\begin{equation} \label{eq:exponentials}
\exp\left(-\beta\int_{\mathbf{R}^2}\ell^X_T(y)^2dy\right)
  \exp\left(-\frac{1}{2}\int_{0}^{T}|X'(t)|^2dt\right)
\end{equation}
where the second exponential is, roughly speaking, the probability 
of a Brownian motion path whose increments are independent normal 
variables.  Here we take $\ell_t^X$ to be the usual local time for 
$X$.  

Large deviation theory also suggests that the overall most likely 
path is realized when the two exponential terms are comparable.  
Now suppose that the walk $X$ is confined to a ball $B_R$ centered 
at the origin and of radius $R$.  The first exponential in 
\eqref{eq:exponentials} will be minimized when $\ell_T^X(y)$ is 
constant over $y\in B_R$, meaning that $\ell_T^X(y)=T/(\pi R^2)$.  
\begin{equation} \label{eq:appendix-ell}
\int_{B_R}\ell^X_T(y)^2dy=\frac{CT^2}{R^4}\cdot R^2=\frac{CT^2}{R^2}.
\end{equation}

As for the second exponential term in \eqref{eq:exponentials}, if 
$X$ reaches the boundary of $B_R$ (as it must for $\ell_T^X$ to be 
constant on $B_R$), then the most likely path has constant velocity,
$X(t)=(R/T)t$.  Then 
\begin{equation} \label{eq:appendix-X}
\int_{0}^{T}|X'(t)|^2dt=T\cdot\left(\frac{R}{T}\right)^2=\frac{R^2}{T}.
\end{equation}

Equating \eqref{eq:appendix-ell} and \eqref{eq:appendix-X}, we get 
\begin{equation*}
R=CT^{3/4}
\end{equation*}
which yields $\nu=3/4$ as conjectured (see \eqref{conj-nu}). 

We have ignored the inconvenient fact that the minimizing paths for 
the two exponential terms are different.  

\subsection{An intuitive justification for Theorem \ref{th:main}}
\label{subsec:thm}

The intuition is quite close to that in Section \ref{subsec:Flory}.
Indeed, in the exponential term involving local time we again assume 
that $\ell_t(y)$ is constant over $y\in[-R,R]$ and so 
$\ell_t(y)=J/(2R)$.  Then,
\begin{equation} \label{eq:appendix-ell-u}
\int_{0}^{T}\int_{-R}^{R}\ell^2_t(y)dydt
= 2TR\left(\frac{J}{2R}\right)^2=\frac{CTJ^2}{R}.
\end{equation}

The main difference is that the approximate probability for a white 
noise $\dot{W}$, intuitively speaking, is 
$\exp(-\frac{1}{2}\int_{0}^{T}\int_{0}^{J}(\dot{W}(t,x))^2dxdt)$, 
since we think of $(\dot{W}(t,x))_{t\in[0,T],x\in[0,J]}$ as a collection of
independent Gaussian variables.  But using \eqref{eq:edwards-wilkinson} to
substitute for $\dot{W}$, we get an approximate probability of 
\begin{equation*}
\exp\left(-\frac{1}{2}\int_{0}^{T}\int_{0}^{J}
  \left[(\partial_tu-\partial_x^2u)(t,x)\right]^2dxdt\right).
\end{equation*}
In such problems the minimizer $u$ is often constant in $t$, giving us
\begin{equation*}
\exp\left(-\frac{T}{2}\int_{0}^{J}\left[\partial_x^2u(t,x)\right]^2dx\right).
\end{equation*}
We might think that the minimizer $u$ has a constant value of 
$|\partial_x^2u|$, consistent with the Neumann boundary conditions.  
Such a function could be 
\begin{equation*}
u(x)=
\begin{cases}
ax^2- aJ^2/4 & \text{if $x\in[0,J/2]$}, \\
a(J-x)^2-aJ^2/4 & \text{if $x\in[J/2,J]$} .
\end{cases}
\end{equation*}
If we require $u(0)=R$ and $u(J)=-R$ then we get $a=4R/J^2$ and 
$|u''(x)|=2a=8R/J^2$.  Then
\begin{equation} \label{eq:appendix-u}
\frac{T}{2}\int_{0}^{J}\left[\partial_x^2u(t,x)\right]^2dx
= CTJ\left(\frac{R}{J^2}\right)^2
= \frac{CTR^2}{J^3}.
\end{equation}
Equating \eqref{eq:appendix-ell-u} and \eqref{eq:appendix-u}, we get
\begin{equation*}
R=CJ^{5/3}
\end{equation*}
which gives the dependence of $R$ on $J$ in Theorem \ref{th:main}.

\subsection{Intuition behind the Conjecture}

The intuition in this case is almost the same as in Section 
\ref{subsec:thm}.  Again, the approximate probability of a path is 
given as a product of two exponentials as in \eqref{eq:exponentials}. 
The second exponential is the same as in \eqref{eq:appendix-u}, 
giving an exponent of 
\begin{equation} \label{eq:conj-u}
\frac{CTR^2}{J^3}
\end{equation}

For the local time exponential, we are using a two dimensional ball 
$B_R$ of volume $CR^2$, so we get an exponent of 
\begin{equation} \label{eq:appendix-ell-2-dim}
\int_{0}^{T}\int_{-R}^{R}\ell^2_t(y)dy
= CTR^2\left(\frac{J}{R^2}\right)^2=\frac{CTJ^2}{R^2}.
\end{equation}

Equating the terms \eqref{eq:conj-u} and \eqref{eq:appendix-ell-2-dim}, 
we get
\begin{equation*}
R=CJ^{5/4}
\end{equation*}
as in Conjecture \ref{conj:dim-2}.



\end{document}